\documentclass[11pt]{article}

\usepackage[letterpaper,margin=1in]{geometry}

\usepackage{amsmath}
\usepackage{amsbsy}
\usepackage{amsfonts}
\usepackage{amssymb}
\usepackage{amscd}
\usepackage{mathrsfs}
\usepackage{bm}
\usepackage{algpseudocode}

\usepackage{txfonts}
\usepackage{lmodern}
\usepackage{xcolor}

\usepackage[latin1]{inputenc}
\usepackage{graphicx}
\usepackage{subfig} 
\usepackage{float}

\usepackage{hyperref}
\usepackage{type1cm}       

\usepackage{soul}
\usepackage{url}

\usepackage{rotating} 
\usepackage{makeidx}
\makeindex             
\usepackage{multicol}   
\usepackage{enumerate}
\usepackage{xspace}
\usepackage{comment}  
\usepackage{etoolbox}
\usepackage{fancyhdr}

\usepackage{array}

\numberwithin{equation}{section}
\numberwithin{figure}{section}
\numberwithin{table}{section}

\ifcsmacro{theorem}{}{
\newtheorem{theorem}{Theorem}[section]
}
\ifcsmacro{lemma}{}{
\newtheorem{lemma}[theorem]{Lemma}
}
\ifcsmacro{corollary}{}{

}
\ifcsmacro{proposition}{}{

}
\ifcsmacro{algorithm}{}{

}

\newtheorem{remark}[theorem]{Remark}
\newtheorem{definition}[theorem]{Definition}

\newcommand{\balpha}{{\boldsymbol{\alpha}}}
\newcommand{\bbeta}{{\boldsymbol{\beta}}}
\newcommand{\be}{{\boldsymbol{e}}}

\begin{document}

\title{$\mathcal{P}_m$ Interior Penalty Nonconforming Finite Element
Methods for $2m$-th Order PDEs in $\mathbb{R}^n$}
\author{
Shuonan Wu\footnote{snwu@math.pku.edu.cn, School of Mathematical Sciences,
Peking University, China, 100871.}\qquad\qquad 
Jinchao Xu\footnote{xu@math.psu.edu, Department of Mathematics,
Pennsylvania State University, University Park, PA, 16802, USA}
}
\date{}

\maketitle

\begin{abstract}
In general $n$-dimensional simplicial meshes, we propose a family of interior penalty nonconforming finite element methods for $2m$-th order partial differential equations, where $m \geq 0$ and $n \geq 1$. For this family of nonconforming finite elements, the shape function space consists of polynomials with a degree not greater than $m$, making it minimal. This family of finite element spaces exhibits natural inclusion properties, analogous to those in the corresponding Sobolev spaces in the continuous case. By applying interior penalty to the bilinear form, we establish quasi-optimal error estimates in the energy norm. Due to the weak continuity of the nonconforming finite element spaces, the interior penalty terms in the bilinear form take a simple form, and an interesting property is that the penalty parameter needs only to be a positive constant of $\mathcal{O}(1)$. These theoretical results are further validated by numerical tests.
\end{abstract}

\section{Introduction}

Let $\Omega$ be a bounded polyhedral domain of $\mathbb{R}^n$. In this paper, we consider the following $2m$-th order partial differential equation:
\begin{equation} \label{equ:m-harmonic}
\left\{
\begin{aligned}
(-\Delta)^m u &= f \quad \mbox{in }\Omega, \\
\frac{\partial^k u}{\partial \nu^k} &= 0 \quad \mbox{on }\partial \Omega, \quad 0 \leq k \leq m-1.
\end{aligned}
\right.
\end{equation}
Here, $f \in L^2(\Omega)$, and $\nu$ denotes the unit outer normal vector on the boundary $\partial \Omega$. The variational formulation of \eqref{equ:m-harmonic} is to find $u \in H_0^m(\Omega)$ such that
$$ 
a(u,v) = (f,v) \quad \forall v \in H_0^m(\Omega),
$$
where 
\begin{equation} \label{equ:bilinear-form}
a(w,v) := (\nabla^m w, \nabla^m v) = \sum_{|\balpha|=m}  \int_\Omega b_{\balpha}
\partial^\balpha w \, \partial^\balpha v \qquad \forall w,v \in H^m(\Omega),
\end{equation}
and $b_{\balpha} = \frac{m!}{\alpha_1! \alpha_2! \cdots \alpha_n!}$. Here and throughout this paper, we use the standard notation for Sobolev spaces as in \cite{ciarlet1978finite, brenner2007mathematical}. For an $n$-dimensional multi-index $\balpha = (\alpha_1, \cdots, \alpha_n)$, we define 
$$ 
|\balpha| = \sum_{i=1}^n \alpha_i, \quad 
\partial^\balpha = \frac{\partial^{|\balpha|}}{\partial x_1^{\alpha_1} \cdots \partial x_n^{\alpha_n}}.
$$

Conforming finite element methods for \eqref{equ:m-harmonic} involve $C^{m-1}$ finite elements, which can lead to extremely complicated constructions when $m \geq 2$ or $n \geq 2$. In \cite{bramble1970triangular}, Bramble and Zl{\'a}mal proposed the 2D simplicial $H^m$ conforming elements ($m \geq 1$) using polynomial spaces of degree $4m-3$, generalizing the $C^1$ Argyris element (cf. \cite{ciarlet1978finite}) and the $C^2$ {\v{Z}}en{\'\i}{\v{s}}ek element (cf. \cite{vzenivsek1970interpolation}). The 3D $C^1$ elements were reported in \cite{zhang2009family, walkington2014c}. Recently, Hu, Lin, and Wu \cite{hu2023construction} developed a construction of conforming elements for $m, n \geq 1$, solving this longstanding problem in the canonical framework. In their construction, the polynomial degree begins at $(m-1)2^n+1$. Further work \cite{hu2024condition} has since shown that this represents the minimal achievable polynomial degree. A geometric interpretation of this key result can be found in \cite{chen2021geometric}. An alternative approach is constructing conforming finite elements on rectangular meshes for arbitrary $m$ and $n$ (see \cite{hu2015minimal}). For \eqref{equ:m-harmonic}, another strategy is to apply a reduced-order mixed discretization (cf. \cite{schedensack2016new, gallistl2017stable}).

As a universal construction with respect to the dimension, Wang and Xu \cite{wang2013minimal} proposed a family of nonconforming finite elements for \eqref{equ:m-harmonic} on $\mathbb{R}^n$ simplicial grids, with the condition that $m \leq n$. These minimal finite elements (referred to as Morley-Wang-Xu or MWX elements) are simple and elegant, combining simplicial geometry, polynomial spaces, and convergence analysis. However, removing the restriction $m \leq n$ for the MWX elements remains a formidable challenge. In \cite{wu2017nonconforming}, Wu and Xu enriched the $\mathcal{P}_n$ polynomial space by adding $\mathcal{P}_{n+1}$ bubble functions, thus obtaining a family of $H^m$ nonconforming elements when $m = n+1$. With carefully designed degrees of freedom, they proved the unisolvent property by exploiting the similarities between the shape function spaces and the degrees of freedom. In \cite{hu2016canonical}, Hu and Zhang used the full $\mathcal{P}_4$ polynomial space to construct nonconforming elements for the case $m=3$, $n=2$, which has three more degrees of freedom locally compared to the elements in \cite{wu2017nonconforming}. They also employed the full $\mathcal{P}_{2m-3}$ polynomial space for nonconforming finite element approximations when $m \geq 4$, $n=2$. In the case of $n=3$, \cite{hu2020family} constructed several high-order nonconforming elements for the $H^2$ problem.  

Until recently, Li and Wu in \cite{li2024construction} extended the MWX nonconforming elements to arbitrary $m, n \geq 1$, and their discretized bilinear form is essentially the original nonconforming form---the bilinear form is simply the broken one. In this construction, when $m > n$, additional bubble functions must be carefully introduced to ensure unisolvence. Although, when $m > n$, the local shape function space is larger than $\mathcal{P}_m$ (i.e., the minimal space), to the best of the author's understanding, it is difficult to find a smaller space from the pure perspective of nonconforming finite elements.

In this paper, we propose a family of interior penalty nonconforming finite element methods for the model problem \eqref{equ:m-harmonic}. The finite element used is a nonconforming element with a shape function space of $\mathcal{P}_m$, making it minimal. The degrees of freedom are carefully designed to preserve weak continuity to the greatest extent possible. For the case where $m > n$, suitable interior penalty terms are introduced to guarantee convergence. As a balance between weak continuity and the interior penalty, the proposed methods require fewer interior penalty terms than the $C^0$-IPDG methods in \cite{brenner2005c0, gudi2011interior, chen2022c0}. As a balance between weak continuity and the interior penalty, the proposed methods require fewer interior penalty terms than the $C^0$-IPDG methods in \cite{brenner2005c0, gudi2011interior, chen2022c0}. In fact, our bilinear form does not include the cross terms for compatibility that are usually present in traditional IPDG methods. As a simple example, for the case where $m = 3$ and $n = 2$, the method involves finding $u_h \in V_h$ such that
$$ 
(\nabla_h^3 u_h, \nabla_h^3 v_h) + \eta \sum_{F \in \mathcal{F}_h}
h_F^{-5}\int_F \llbracket u_h \rrbracket \cdot \llbracket v_h
\rrbracket = (f, v_h) \qquad \forall v_h \in V_h, 
$$
where the nonconforming element is illustrated in Figure \ref{fig:m3n2}. Furthermore, Figure \ref{fig:m4n2} shows the nonconforming element for $m=4, n=2$, where the proposed method is written as
$$ 
(\nabla_h^4 u_h, \nabla_h^4 v_h) + \eta \sum_{F \in \mathcal{F}_h}
h_F^{-5}\int_F \sum_{|\balpha|=1} \llbracket \partial^{\balpha} u_h
\rrbracket \cdot \llbracket \partial^{\balpha} v_h \rrbracket = (f,
v_h) \qquad \forall v_h \in V_h. 
$$

The advantages of our methods include that the resulting stiffness matrix is symmetric and positive definite. Most importantly, the penalty parameter $\eta$ only needs to be a positive constant of $\mathcal{O}(1)$. Furthermore, the proposed methods are well-suited for handling more complicated nonlinear high-order models (cf. \cite{du2004phase, barrett2004finite, backofen2007nucleation, wise2009energy, hu2009stable, wang2011energy, dai2013geometric, doelman2014meander}).

\begin{figure}[!htbp]
\centering 
\subfloat[$m=3,n=2$]{\centering 
   \includegraphics[width=0.25\textwidth]{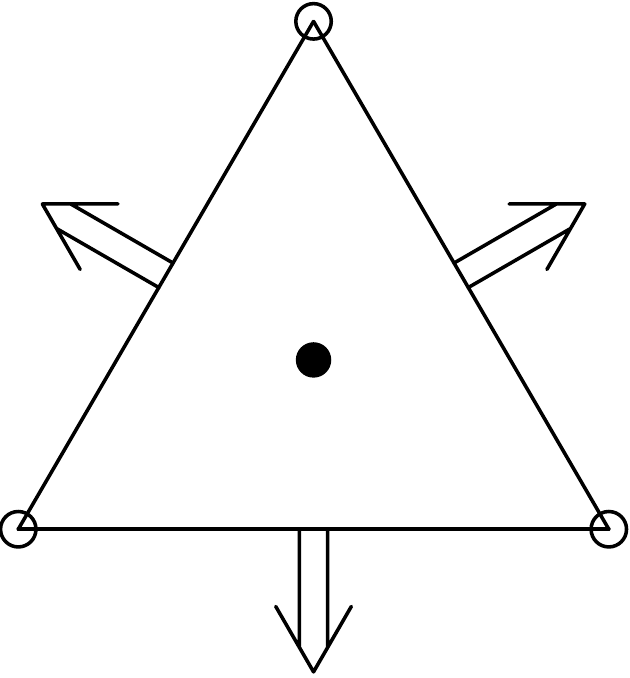} 
   \label{fig:m3n2}
}%
\qquad\qquad  
\subfloat[$m=4,n=2$]{\centering 
   \includegraphics[width=0.25\textwidth]{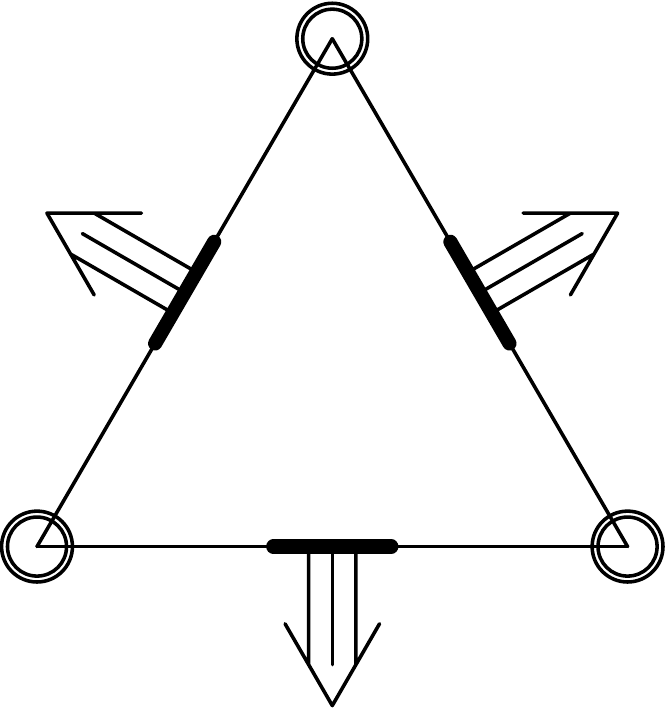} 
   \label{fig:m4n2}
}
\end{figure}

Finally, for virtual element methods (VEM) \cite{beirao2013basic, beirao2014hitchhiker} with local spaces that are not polynomial, there have been many recent advancements in their application to high-order equations. For conforming VEM, we refer to \cite{antonietti2020conforming, da2020c1, chen2022conforming}. As for nonconforming VEM, additional references can be found in \cite{huang2020nonconforming, chen2020nonconforming}.

The rest of the paper is organized as follows. In Section
\ref{sec:spaces}, we provide a detailed description of the family of 
nonconforming finite element spaces for arbitrary $m,n$.  In Section
\ref{sec:IP-nonconforming}, we state the interior penalty
nonconforming methods and prove the well-posedness. In Section
\ref{sec:convergence}, we prove the quasi-optimal error estimates of
the interior penalty nonconforming method provided that the conforming
relatives exist.  We further give the error estimates without the
conforming relatives assumption, but under an extra regularity
assumption.  Numerical tests are provided in Section
\ref{sec:numerical} to support the theoretical findings, and some
concluding remarks will then arrive to close the main text.

\section{The minimal nonconforming finite element spaces} \label{sec:spaces}

Let $\mathcal{T}_h$ be a conforming and shape-regular simplicial triangulation of $\Omega$, and let $\mathcal{F}_h$ denote the set of all faces in $\mathcal{T}_h$. The set of interior faces is $\mathcal{F}_h^i = \mathcal{F}_h \setminus \partial \Omega$, and the boundary faces are given by $\mathcal{F}_h^{\partial} = \mathcal{F}_h \cap \partial \Omega$. Define $h := \max_{T \in \mathcal{T}_h} h_T$, where $h_T$ is the diameter of $T$ (cf. \cite{ciarlet1978finite, brenner2007mathematical}). We assume that $\mathcal{T}_h$ is quasi-uniform, i.e.,  
\[
\exists \theta > 0 \quad \text{such that} \quad \max_{T \in \mathcal{T}_h} \frac{h}{h_T} \leq \theta,
\]
where $\theta$ is a constant independent of $h$. For any $F \in \mathcal{F}_h$, let $\omega_F$ denote the union of all simplices sharing the face $F$. The diameter of $\omega_F$ is denoted by $h_F$. For any $F \in \mathcal{F}_h^i$, there exist two simplices, $T^+$ and $T^-$, such that $F = \partial T^+ \cap \partial T^-$. The unit normals $\nu^+$ and $\nu^-$ on $F$ are defined to point outward from $T^+$ and $T^-$, respectively. For a function $v$, we set its traces on $F$ as $v^{\pm} = v|_{\partial T^{\pm}}$, and define the following standard DG notation:  
\[
[v] = v^+ - v^-, \quad \{v\} = \frac{1}{2}(v^+ + v^-), \quad \llbracket v\rrbracket = v^+ \nu^+ + v^-\nu^- \quad \text{on } F \in \mathcal{F}_h^i.
\] 
On the boundary faces, $v$ has a uniquely defined trace. Thus, we set  
\[
[v] = v, \quad \{v\} = v, \quad \llbracket v \rrbracket = v\nu \quad \text{on } F \in \mathcal{F}_h^{\partial}.
\]
Here, $\llbracket\cdot\rrbracket$ and $\{\cdot\}$ represent the jump and average operators, respectively, following the standard DG notation (cf. \cite{arnold2002unified}).  

Based on the triangulation $\mathcal{T}_h$, for $v \in L^2(\Omega)$ with $v|_T \in H^k(T)$ for all $T \in \mathcal{T}_h$, we define $\partial^{\balpha}_h v$ as the piecewise partial derivatives of $v$ for multi-indices $\balpha$ with $|\balpha| \leq k$. The broken norms and seminorms are then given by:  
\[
\|v\|_{k,h}^2 := \sum_{T \in \mathcal{T}_h}\|v\|_{k,T}^2, \quad
|v|_{k,h}^2 := \sum_{T \in \mathcal{T}_h} |v|_{k,T}^2.
\] 

For convenience, we use $C$ to denote a generic positive constant that may represent different values in different occurrences but is independent of the mesh size $h$. The notation $X \lesssim Y$ means $X \leq CY$.

\subsection{Definition of nonconforming finite elements}
%

For any $m \geq 0$ and $n \geq 1$, we define $L = \lfloor \frac{m}{n+1} \rfloor$, where $\lfloor x \rfloor$ denotes the greatest integer less than or equal to $x$. Following the description in \cite{ciarlet1978finite, brenner2007mathematical}, a finite element is represented by a triple $(T, P_T, D_T)$, where $T$ denotes the geometric shape of the element, $P_T$ is the shape function space, and $D_T$ is the set of degrees of freedom, which is $P_T$-unisolvent. For the $2m$-th order elliptic problem \eqref{equ:m-harmonic}, the minimal shape function space is $P_T^{(m,n)} := \mathcal{P}_m(T)$, where $\mathcal{P}_k(T)$ represents the space of all polynomials on $T$ with degree not exceeding $k$, for any integer $k \geq 0$.

For $0 \leq k \leq n$, let $\mathcal{F}_{T,k}$ denote the set of all $(n-k)$-dimensional sub-simplices of $T$ (thus, $k$ represents the co-dimension). For any $F \in \mathcal{F}_{T,k}$ $(1 \leq k \leq n)$, let $|F|$ denote its $(n-k)$-dimensional measure, and let $\nu_{F,1}, \ldots, \nu_{F,k}$ be linearly independent unit vectors orthogonal to the tangent space of $F$. Specifically, when $k = n$, $F$ represents a vertex and $|F| = 1$.  

For $1\leq k \leq n$, let $A_k$ be the set consisting of all multi-indexes $\balpha$ with $\sum_{i=k+1}^n \alpha_i = 0$. For any $(n-k)$-dimensional sub-simplex $F\in \mathcal{F}_{T,k}$ and $\balpha \in A_k$ with 
$$ 
|\balpha| = m-k-(n+1)(L-\ell), \qquad 0\leq \ell \leq L, 
$$
define 
\begin{equation} \label{equ:DOF-subsimplex}
\begin{aligned}
d_{T,F,\balpha}(v) &:= \frac{1}{|F|} \int_F
\frac{\partial^{|\balpha|}v}{\partial \nu_{F,1}^{\alpha_1} \cdots
\partial \nu_{F,k}^{\alpha_k}} 
\qquad \forall v\in H^m(T),\\
d_{T,0}(v) &:= \frac{1}{|T|} \int_T v \qquad\qquad\qquad\quad ~~\forall
v\in H^m(T).
\end{aligned}
\end{equation}
Here, $\ell$ represents the level of the degrees of freedom. By the Sobolev embedding theorem (cf. \cite{adams2003sobolev}), $d_{T,F,\balpha}$ is a continuous linear functional on $H^{m}(T)$. We define the set of the degrees of freedom at level $\ell$ as 
\begin{equation} \label{equ:DOFs-l}
\begin{aligned}
\tilde{D}_{T,\ell}^{(m,n)} &:=  \Big\{ d_{T,F,\balpha}~|~ \balpha \in A_k
\text{ with } |\balpha| = m-k-(n+1)(L-\ell),\\
& \qquad \qquad F\in \mathcal{F}_{T,k}, 1 \leq k \leq \min\{n,
    m-(n+1)(L-\ell)\} \Big\}, \qquad \text{for } 0 \leq \ell \leq L.
\end{aligned}
\end{equation}

We further define the set of degrees of freedom at level $\ell=-1$ as 
\begin{equation} \label{equ:DOFs--1}
\tilde{D}_{T,-1}^{(m,n)} := 
\left\{
\begin{aligned}
& \{d_{T,0}\} \qquad\qquad \text{if } m
\equiv 0~(\bmod ~n+1), \\
& \varnothing \qquad\qquad\qquad \text{otherwise}.
\end{aligned}
\right.
\end{equation}
Then, the set of the degrees of freedom is 
\begin{equation} \label{equ:DOFs}
D_T^{(m,n)}= \bigcup_{\ell=-1}^L \tilde{D}_{T,\ell}^{(m,n)}. 
\end{equation}
The diagrams of the finite elements ($0 \leq m \leq 5, 1\leq n \leq 3$) are plotted in Table \ref{tab:period-table}.

\begin{table}[!htbp]
\begin{center}
\begin{tabular}{>{\centering\arraybackslash}m{.6cm} |
>{\centering\arraybackslash}m{4.0cm} |
>{\centering\arraybackslash}m{4.0cm} |
>{\centering\arraybackslash}m{4.0cm} @{}m{0pt}@{} }
\hline
$m \backslash n$ & 1 & 2 & 3 \\ \hline
0 &
\includegraphics[width=1.3in]{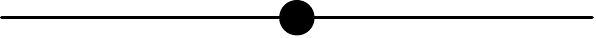}
& 
\includegraphics[width=1.3in]{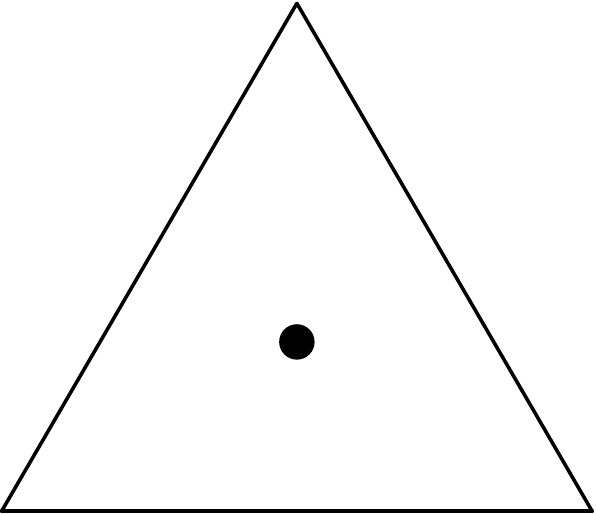}
& 
\includegraphics[width=1.3in]{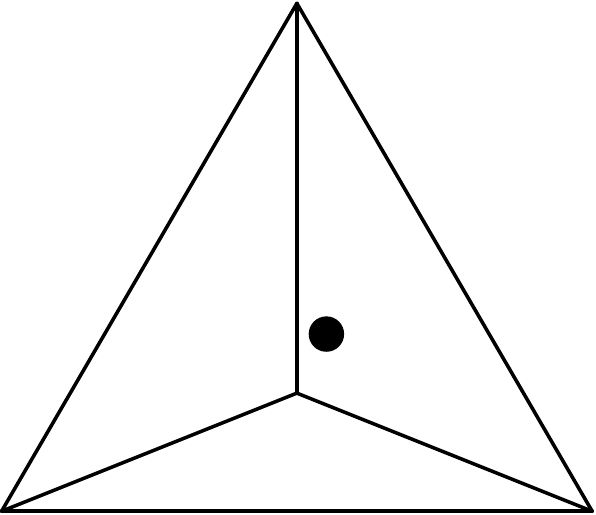} \\ 
\hline 
1 &
\includegraphics[width=1.3in]{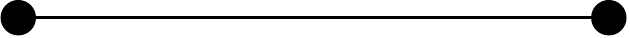}
& 
\includegraphics[width=1.3in]{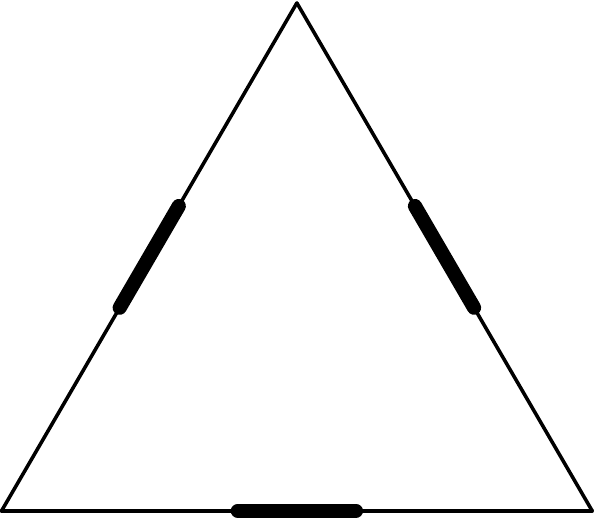}
& 
\includegraphics[width=1.3in]{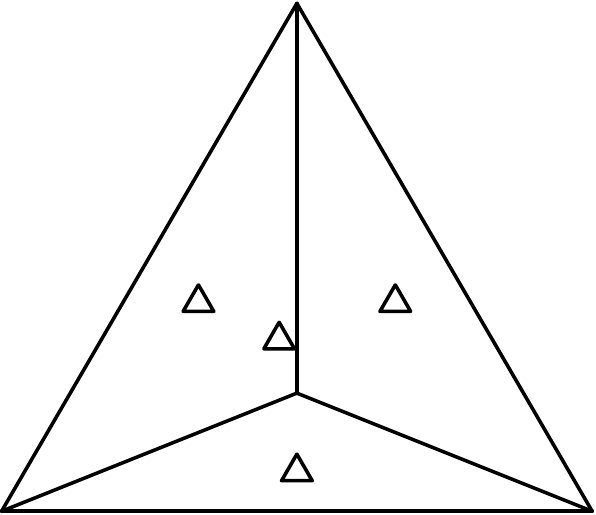} \\ 
\hline 
2 & 
\includegraphics[width=1.3in]{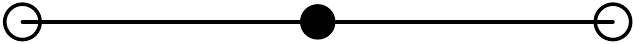}
& 
\includegraphics[width=1.3in]{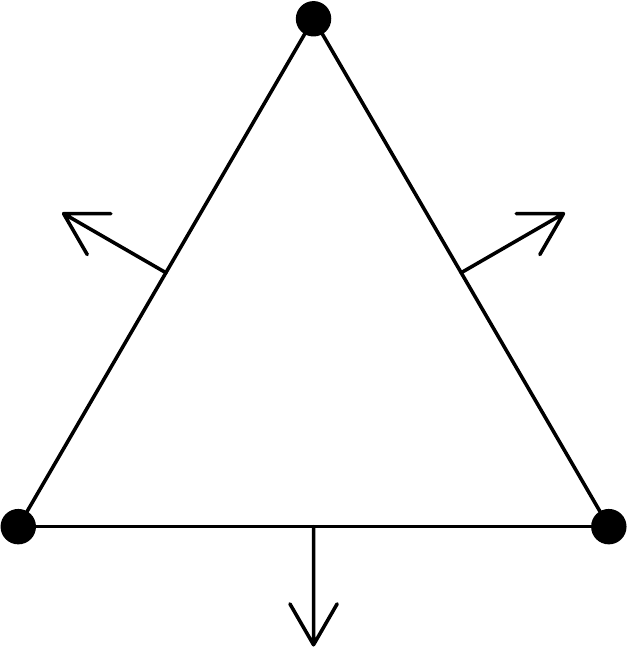}
& 
\includegraphics[width=1.3in]{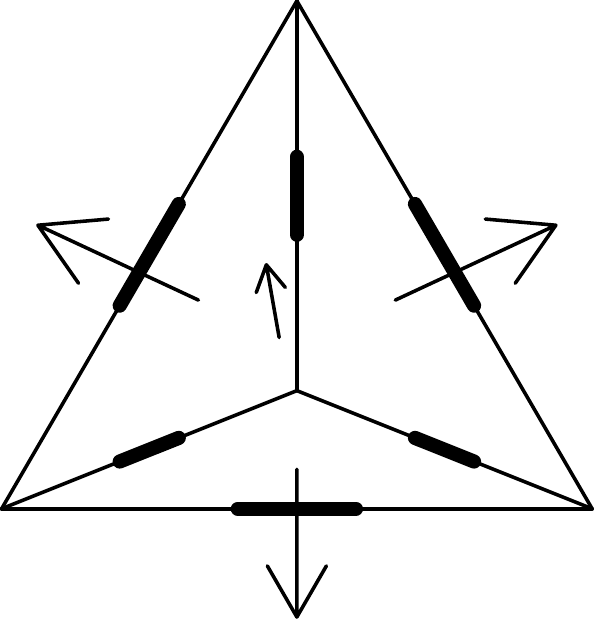}
\\ \hline 
3 & 
\includegraphics[width=1.3in]{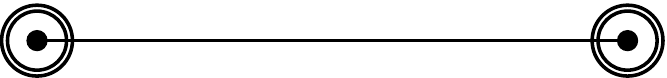}
&
\includegraphics[width=1.3in]{MWX2Dm3-IP.pdf}
&
\includegraphics[width=1.3in]{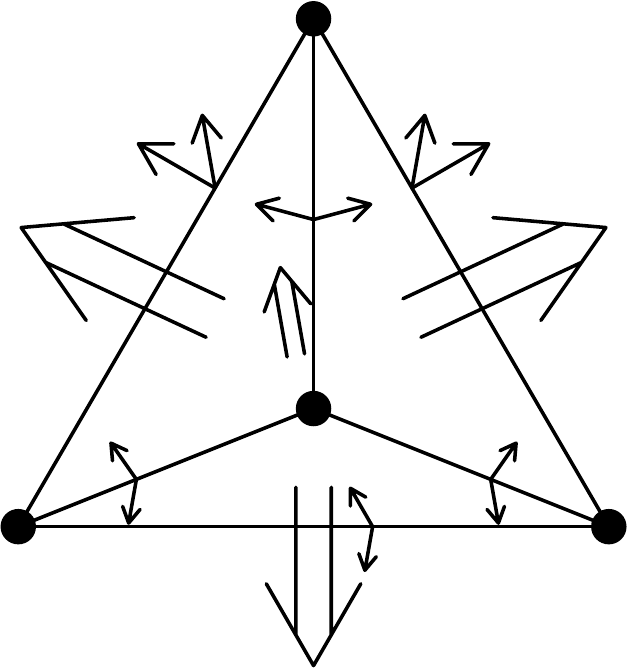}
\\ \hline
4 & 
\includegraphics[width=1.3in]{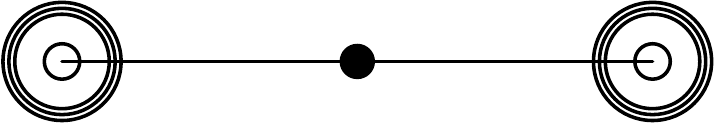}
&
\includegraphics[width=1.3in]{MWX2Dm4-IP.pdf}
&
\includegraphics[width=1.3in]{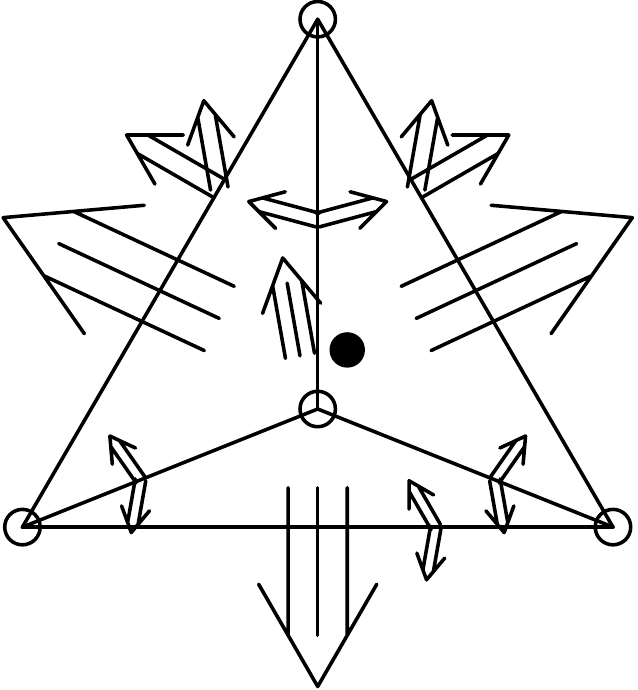}
\\ \hline 
5 & 
\includegraphics[width=1.3in]{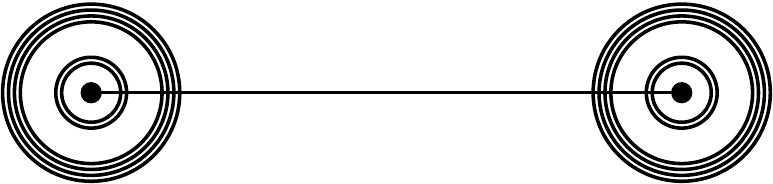}
&
\includegraphics[width=1.3in]{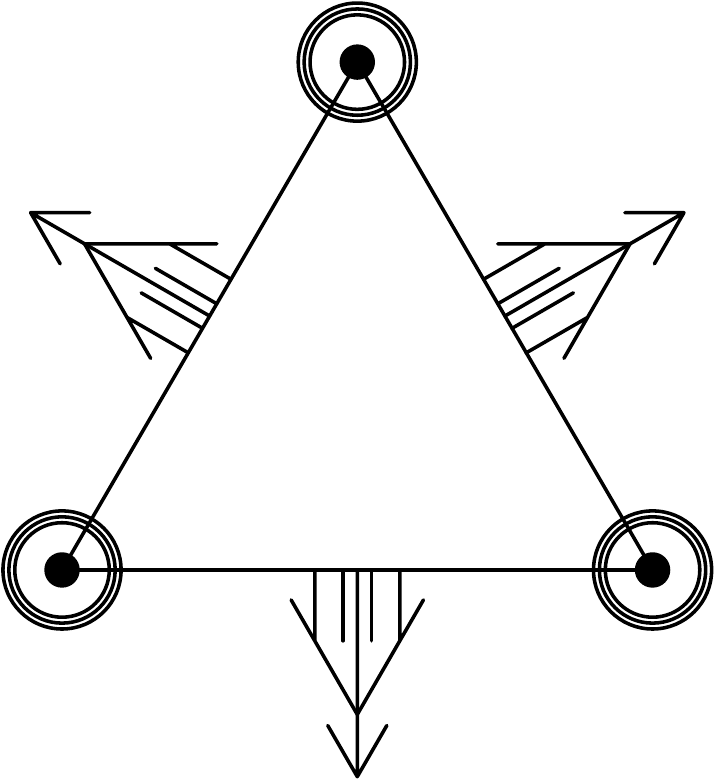}
&
\includegraphics[width=1.3in]{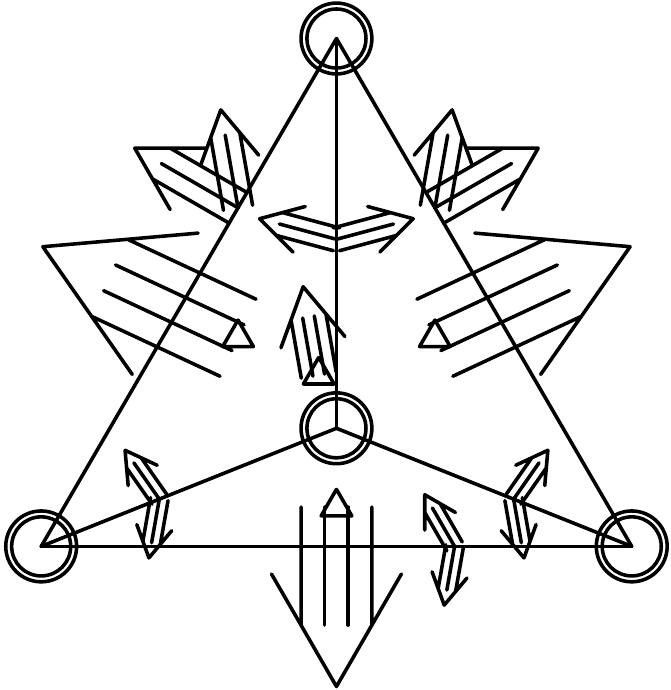}
\\ \hline 
\end{tabular}
\caption{Degrees of freedom: $0 \leq m \leq 5$, $1 \leq n \leq 3$}
\label{tab:period-table}
\end{center}
\end{table}

\begin{remark}[extension of the Morley-Wang-Xu elements]
It can be observed that when $1 \leq m \leq n$, or equivalently when $L = 0$, the proposed finite elements reduce to the Morley-Wang-Xu elements as introduced in \cite{wang2013minimal}.
\end{remark}

We number the local degrees of freedom by 
$$ 
d_{T,1}, d_{T,2}, \cdots, d_{T,J^{(m,n)}},
$$ 
where $J^{(m,n)}$ is the number of local degrees of freedom. 

\begin{lemma}[number of degrees of freedom]
For any $m\geq 0, n\geq 1$, $J^{(m,n)} = {\rm dim} P_T^{(m,n)}$.
\end{lemma}
\begin{proof}
Let the combinatorial number $C_j^i = \frac{j!}{i!(j-i)!}$ for $j \geq i$, and $C_j^i = 0$ for $j < i$ or $j < 0$. For each $0 \leq k \leq n$, $T$ has $C_{n+1}^{k}$ sub-simplices of dimension $n-k$. For each $(n-k)$-dimensional sub-simplex $F$, the number of all $(m-k-(n+1)(L-\ell))$-th order directional derivatives with respect to $\nu_{F,1}, \ldots, \nu_{F,k}$ is $C_{m-(n+1)(L-\ell)-1}^{k-1} = C_{m-(n+1)(L-\ell)-1}^{m-(n+1)(L-\ell)-k}$. Using the fact that $0 \leq m-(n+1)\lfloor \frac{m}{n+1} \rfloor \leq n$, we have
$$
\begin{aligned}
J^{(m,n)} &= \left| \tilde{D}_{T,-1}^{(m,n)} \right| + \sum_{\ell=0}^L 
\left|\tilde{D}_{T,\ell}^{(m,n)} \right| \\
&= \left| \tilde{D}_{T,-1}^{(m,n)} \right| + \sum_{\ell=0}^L 
\sum_{k=1}^{\min\{n,m-(n+1)(L-\ell)\}} C_{n+1}^k
C_{m-(n+1)(L-\ell)-1}^{m-(n+1)(L-\ell)-k} \\
&= \left| \tilde{D}_{T,-1}^{(m,n)} \right| + \sum_{\ell=0}^L 
\sum_{k=-\infty}^{\min\{n,m-(n+1)(L-\ell)\}} C_{n+1}^k
C_{m-(n+1)(L-\ell)-1}^{m-(n+1)(L-\ell)-k} \\
&= \left| \tilde{D}_{T,-1}^{(m,n)} \right| 
+ \sum_{k=-\infty}^{m-(n+1)L} C_{n+1}^k
C_{m-(n+1)L-1}^{m-(n+1)L-k}   
+ \sum_{\ell=1}^L \sum_{k=-\infty}^{n} C_{n+1}^k
C_{m-(n+1)(L-\ell)-1}^{m-(n+1)(L-\ell)-k} \\ 
&= \left| \tilde{D}_{T,-1}^{(m,n)} \right| 
+ \sum_{k=-\infty}^{m-(n+1)L} C_{n+1}^k C_{m-(n+1)L-1}^{m-(n+1)L-k}   
+ \sum_{\ell=1}^L \left[ C_{m-(n+1)(L-1-\ell)-1}^{m-(n+1)(L-\ell)} -
C_{m-(n+1)(L-\ell)-1}^{m-(n+1)(L+1-\ell)} \right] \\ 
&= \left| \tilde{D}_{T,-1}^{(m,n)} \right| + \left(
\sum_{k=-\infty}^{m-(n+1)L} C_{n+1}^k
C_{m-(n+1)L-1}^{m-(n+1)L-k} \right) +  
C_{m+n}^n - C_{m-(n+1)(L-1)-1}^n.
\end{aligned}
$$ 

Note that $m-(n+1)L = 0$ when $m \equiv 0~(\bmod ~n+1)$. Hence, 
$$ 
\sum_{k=-\infty}^{m-(n+1)L} C_{n+1}^k
C_{m-(n+1)L-1}^{m-(n+1)L-k} = \left\{
\begin{aligned}
&0 \qquad\qquad\qquad\qquad\quad\text{if } m\equiv 0 ~ (\bmod ~n+1), \\
&C_{m-(n+1)(L-1)-1}^n \qquad~ \text{otherwise}.
\end{aligned}
\right. 
$$ 
In conjunction with the definition of $\tilde{D}_{T,-1}^{(m,n)}$
in \eqref{equ:DOFs--1}, we conclude that 
$$ 
J^{(m,n)} = C_{m+n}^n \qquad \forall m \geq 0, n\geq 1,
$$ 
which is exactly the dimension of $P_T^{(m,n)}$.
\end{proof}

\subsection{Unisolvent property}
We shall show the $P_T$-unisolvent property of the new family of finite elements.

\begin{lemma}[integrals of function derivatives on sub-simplices] \label{lem:sub-simplex-derivatives}  
For any $0 \leq \ell \leq L$, let $0 \leq k \leq \min\{n, m-(n+1)(L-\ell)\}$ and $F \in \mathcal{F}_{T,k}$. Then, for any $v \in H^m(T)$, the integrals of all its $(m-k-(n+1)(L-\ell))$-th order derivatives on $F$, i.e.,  
$$  
\int_F \partial^{\balpha}v, \qquad |\balpha| = m-k-(n+1)(L-\ell),  
$$  
are determined by $\tilde{D}_{T,\ell}^{(m,n)}$ given in \eqref{equ:DOFs-l}. In other words, if all the degrees of freedom in $\tilde{D}_{T,\ell}^{(m,n)}$ are zero, then the above integrals vanish.  
\end{lemma}  

\begin{proof}
we prove the lemma by induction. When $k = \min\{n, m-(n+1)(L-\ell)\}$, we consider the following two cases:
\begin{enumerate}
\item If $k = n$, namely $\mathcal{F}_{T,k} =
\mathcal{F}_{T,n}$ will be the set of vertices. Therefore,
$\{\nu_{F,1}, \ldots, \nu_{F,n}\}$ forms a set of basis of
$\mathbb{R}^n$. The lemma holds trivially since the $|\balpha|$-th order derivatives at the vertices are included in $\tilde{D}_{T,\ell}^{(m,n)}$.

\item If $k = m-(n+1)(L-\ell) \leq n$, namely $\ell=0$. In this case, we have $|\balpha| = 0$,
$$ 
\frac{1}{|F|}\int_F v = d_{T,F,0}(v). 
$$ 
The lemma is also true.
\end{enumerate}

Assume that the lemma holds for all $i+1 \leq k \leq \min\{n, m-(n+1)(L-\ell)\}$.  
Now, we consider the case where $k = i$. Denote by $S_1, S_2, \ldots, S_{n-k-1}$ all $(n-k-1)$-dimensional sub-simplices of the $(n-k)$-simplex $F$, and let $\nu^{(j)}$ denote the unit outer normal vector of $S_j$, viewed as part of the boundary of the $(n-k)$-simplex in $(n-k)$-dimensional space. Choose orthogonal unit vectors $\tau_{F,k+1}, \ldots, \tau_{F,n}$ that are tangent to $F$. Then, the set  
$\{\nu_{F,1}, \ldots, \nu_{F,k}, \tau_{F,k+1}, \ldots, \tau_{F,n}\}$ forms a basis of $\mathbb{R}^n$.

For any $|\balpha| = m-k-(n+1)(L-\ell)$, if $\alpha_{k+1} = \cdots = \alpha_n = 0$, then
$$
\frac{1}{|F|} \int_F \frac{\partial^{|\balpha|}v}{\partial
\nu_{F,1}^{\alpha_1} \cdots \partial \nu_{F,k}^{\alpha_k} \partial
\tau_{F,k+1}^{\alpha_{k+1}} \cdots \tau_{F,n}^{\alpha_n}} = 
d_{T,F,|\balpha|}(v).
$$
Otherwise, without loss of generality, assume that $\alpha_{k+1} > 0$. By applying Green's formula, we obtain
$$
\int_F \frac{\partial^{|\balpha|}v}{\partial
\nu_{F,1}^{\alpha_1} \cdots \partial \nu_{F,k}^{\alpha_k} \partial
\tau_{F,k+1}^{\alpha_{k+1}} \cdots \tau_{F,n}^{\alpha_n}} 
= \sum_{j=1}^{n-k+1} \nu^{(j)} \cdot \tau_{F,k+1} \int_{S_j} 
\frac{\partial^{|\balpha|-1}v}{\partial
\nu_{F,1}^{\alpha_1} \cdots \partial \nu_{F,k}^{\alpha_k} \partial
\tau_{F,k+1}^{\alpha_{k+1}-1} \cdots \tau_{F,n}^{\alpha_n}},
$$
where the integrals over the sub-simplices $S_j$ can be determined by $\tilde{D}_{T,\ell}^{(m,n)}$ by the induction hypothesis. Thus, the lemma holds for $k=i$. This completes the induction argument.
\end{proof}

\begin{lemma}[inclusion property] \label{lem:MWX-derivatives}
For any $v \in P_T^{(m,n)}$ having all the degrees of freedom zero,
and $|\bbeta| \leq m$, we have 
\begin{equation} \label{equ:derivatives-DOF}
d_T(\partial^{\bbeta} v) = 0
\qquad \forall d_T \in D_T^{(m-|\bbeta|,n)}. 
\end{equation}
\end{lemma}
\begin{proof}
By induction, it suffices to prove the case when $|\bbeta| = 1$. We divide the proof into two cases:

\begin{enumerate}
\item If $m \not\equiv 0~(\bmod ~n+1)$, then $L = \lfloor \frac{m}{n+1} \rfloor = \lfloor \frac{m-1}{n+1} \rfloor$. For any $d_{T,F,\balpha} \in \tilde{D}_{T,\ell}^{(m-1,n)}$ (see \eqref{equ:DOFs-l}) with $0 \leq \ell \leq L$, $F \in \mathcal{F}_{T,k}$, and $|\balpha| = m-1-k-(n+1)(L-\ell)$, we have 
$$
d_{T,F,\balpha} (\partial^{\bbeta}v) = \frac{1}{|F|} \int_F \partial^{\bbeta} \frac{\partial^{|\balpha|}v}{\partial \nu_{F,1}^{\alpha_1} \cdots \partial \nu_{F,k}^{\alpha_k}}.
$$ 
Notice that $|\balpha| + |\bbeta| = m-k-(n+1)(L-\ell)$, and thus $d_{T,F,\balpha}(\partial^{\bbeta}v) = 0$ by Lemma \ref{lem:sub-simplex-derivatives}. If further $m = (n+1)L + 1$ so that $\tilde{D}_{T,-1}^{(m-1,n)}$ defined in \eqref{equ:DOFs--1} is non-empty, by setting $\ell = L$ and $k=0$ in Lemma \ref{lem:sub-simplex-derivatives}, we obtain $d_{T,0}(\partial^{\bbeta}v) = \frac{1}{|T|} \int_T \partial^{\bbeta}v = 0$, since $\tilde{D}^{(m,n)}_{T,L}$ vanishes.

\item If $m \equiv 0~(\bmod ~n+1)$, then $\lfloor \frac{m-1}{n+1} \rfloor = \lfloor \frac{m}{n+1} \rfloor - 1 = L-1$. For any $d_{T,F,\balpha} \in \tilde{D}^{(m-1,n)}_{T,\ell}$ with $0 \leq \ell \leq L-1$, $F \in A_k$, and $|\balpha| = m-1-k-(n+1)\left(\lfloor \frac{m-1}{n+1} \rfloor - \ell\right)$, we have
$$
|\balpha| + |\bbeta| = m-k-(n+1)[L-(\ell+1)],
$$
which implies that $d_{T,F,\balpha}(\partial^{\bbeta}v) = 0$ since all the degrees of freedom in $\tilde{D}_{T,\ell+1}^{(m,n)}$ vanish.
\end{enumerate}
This completes the proof.
\end{proof}

Thanks to the above two lemmas, we can prove the unisolvent property of
the new family of nonconforming finite elements. 

\begin{theorem}[unisolvence]\label{thm:unisolvent}
For any $n \geq 1$, $m \geq 0$, $D_T^{(m,n)}$ is
$P_T^{(m,n)}$-unisolvent. 
\end{theorem}
\begin{proof}
As the dimension of $P_T^{(m,n)}$ is the same as the number of
local degrees of freedom, it suffices to show that $v=0$ if all the
degrees of freedom vanish. 

Notice that for any $v \in P_T^{(m,n)} = \mathcal{P}^m(T)$, $\partial_i v
\in \mathcal{P}^{m-1}(T) = P_T^{(m-1,n)}(T)$. Then, by induction and Lemma
\ref{lem:MWX-derivatives} (inclusion property), we obtain
that $\partial^{\bbeta} v = 0$ for all $|\bbeta| = 1$, which implies that $v$ is a constant. If $m \equiv 0~(\bmod~n+1)$, we have $v = 0$ by $d_{T,0}(v) = 0$. Otherwise, by
Lemma \ref{lem:sub-simplex-derivatives},
$$ 
\int_F v = 0 \qquad \forall F \in \mathcal{F}_{T,k},~ k = m - (n+1)L,
$$ 
which also implies that $v = 0$. Then, we finish the proof.
\end{proof}

\subsection{Canonical nodal interpolation and global finite element spaces}
Building on Theorem \ref{thm:unisolvent}, we define the interpolation operator $\Pi_T^{(m,n)}: H^{m}(T) \to P_T^{(m,n)}$ by

\begin{equation} \label{equ:interpolation}
\Pi_T^{(m,n)} v := \sum_{i=1}^{J^{(m,n)}} p_i d_{T,i}(v) \quad \forall v \in H^{m}(T),
\end{equation}
where $p_i \in P_T^{(m,n)}$ are the nodal basis functions that satisfy $d_{T,j}(p_i) = \delta_{ij}$, with $\delta_{ij}$ denoting the Kronecker delta. It is important to note that the operator $\Pi_T^{(m,n)}$ is well-defined for all functions in $H^{m}(T)$. The error estimate for this interpolation operator can be derived using standard interpolation theory (cf. \cite{ciarlet1978finite, brenner2007mathematical}).

\begin{lemma}[local interpolation error] \label{lem:interpolation}
For $s \in [0,1]$, it holds that, for any integer $0\leq k \leq m$,  
\begin{equation} \label{equ:interpolation-error}
|v - \Pi_T^{(m,n)} v|_{k,T} \lesssim h_T^{m+s-k}|v|_{m+s,T} \qquad \forall v\in H^{m+s}(T),
\end{equation}
for all shape-regular $n$-simplex $T$.
\end{lemma}

We define the piecewise polynomial spaces $V_h^{(m,n)}$ and $V_{h0}^{(m,n)}$ as follows:
\begin{itemize}
\item $V_h^{(m,n)}$ consists of all functions $v_h|_T \in P_{T}^{(m,n)}$, such that for any $0 \leq l \leq L$, any $(n-k)$-dimensional sub-simplex $F$ of any $T \in \mathcal{T}_h$ with $1 \leq k \leq \min\{n,m-(n+1)(L-l)\}$, and any $\balpha \in A_k$ with $|\balpha| = m-k-(n+1)(L-l)$, the functional $d_{T,F,\balpha}(v)$ is continuous.

\item $V_{h0}^{(m,n)} \subset V_h^{(m,n)}$ is defined such that for any $v_h \in V_{h0}^{(m,n)}$, we have $d_{T,F,\balpha}(v_h) = 0$ if the $(n-k)$-dimensional sub-simplex $F \subset \partial \Omega$.
\end{itemize}

The global interpolation operator $\Pi_h^{(m,n)}$ on $H^{m}(\Omega)$
is defined as follows: 
\begin{equation} \label{equ:interpolation-global}
(\Pi_h^{(m,n)} v)|_T := \Pi_T^{(m,n)}(v|_T) \qquad \forall T\in
\mathcal{T}_h, v \in H^{m}(\Omega).
\end{equation}
By the above definition, we have $\Pi_h^{(m,n)}v \in V_h^{(m,n)}$ for
any $v\in H^m(\Omega)$ and $\Pi_h^{(m,n)}v \in V_{h0}^{(m,n)}$ for any
$v\in H_0^m(\Omega)$. 
%

\subsection{Weak continuity}
At the end of this section, we explore the weak continuity of the nonconforming element spaces.

\begin{definition}[$p$-weak continuity \& $p$-weak zero-boundary condition] \label{def:weak-continuity} 
The finite element space \(V_h\) is said to satisfy the $p$-weak continuity if, for any \(F \in \mathcal{F}_h^i\), any \(v_h \in V_h\), and any multi-index \(\balpha\) with \(|\balpha| = p\), the derivative \(\partial_h^{\balpha} v_h\) is continuous at least at one point on \(F\). Moreover, \(V_{h0}\) satisfies the \(k\)-weak zero-boundary condition if, for any \(F \subset \mathcal{F}_h^{\partial}\), any \(v_h \in V_{h0}\), and any multi-index \(\balpha\) with \(|\balpha| = p\), the derivative \(\partial_h^{\balpha} v_h\) vanishes at least at one point on \(F\).
\end{definition}

Using the intermediate value theorem, the definitions of $p$-weak continuity and $p$-weak zero-boundary condition straightforwardly yield the following lemma.
\begin{lemma}[property of $p$-weak continuity and $p$-weak zero-boundary condition] \label{lem:k-weak-continuity}
If $V_h$ has $p$-weak continuity, then for any $v_h \in V_h$ and any $F \in \mathcal{F}_{h}^i$,
\begin{equation} \label{equ:k-weak-continuity}
\max_{|\balpha|=p}\max_{x \in F} |\partial_h^{\balpha} v_h^T(x) - \partial_h^{\balpha} v_h^{T'}(x)| \lesssim h_F \max_{y \in F} \sum_{|\bbeta| = k+1} |\partial^{\bbeta} v_h^T(y) - \partial^{\bbeta} v_h^{T'}(y)|.
\end{equation}
Furthermore, if $V_{h0}$ satisfies $p$-weak zero-boundary condition, then for any $v_h \in V_{h0}$, the following hold:
\begin{equation} \label{equ:k-weak-continuity}
\max_{|\balpha|=p} \max_{x \in F} |\partial_h^{\balpha} v_h^T(x)| \lesssim h_F \max_{y \in F} \sum_{|\bbeta| = p+1} |\partial^{\bbeta} v_h^T(y)| \qquad \forall F \in \mathcal{F}_h^\partial.
\end{equation}
\end{lemma}

We now examine the weak continuity and weak zero-boundary condition of the finite element spaces proposed in this paper. By Lemma \ref{lem:sub-simplex-derivatives} (integrals of function derivatives on sub-simplices), we can directly obtain the following lemma.

\begin{lemma} \label{lem:face-average}
For any $0 \leq \ell \leq L$, let $1 \leq k \leq \min\{m-(n+1)(L-\ell)\}$, and let $F$ be an $(n-k)$-dimensional sub-simplex of $T \in \mathcal{T}_h$. Then, for any $v_h \in V_h^{(m,n)}$ and any $T' \in \mathcal{T}_h$ with $F \subset T'$, we have
\begin{equation} \label{equ:face-average}
\int_F \partial^{\balpha}(v_h|_{T}) = \int_F \partial^{\balpha}(v_h|_{T'}) \qquad |\balpha| = m-k-(n+1)(L-\ell).
\end{equation}
If $F \subset \partial \Omega$, then for any $v_h \in V_{h0}^{(m,n)}$, it holds that
\begin{equation} \label{equ:face-average0}
\int_F \partial^{\balpha}(v_h|_{T}) = 0 \qquad |\balpha| = m-k-(n+1)(L-\ell).
\end{equation}
\end{lemma}

For the $2m$-order elliptic problem, only the weak continuity for $p < m$ is required. From this lemma, we observe that $V_h$ (resp. $V_{h0}$) satisfies the $p$-weak continuity (resp. $p$-weak zero-boundary condition), except for $p = m - (n+1)(L-\ell+1)$, for any $1 \leq \ell \leq L$ (i.e., $k$ in the lemma cannot be zero). To address this, interior penalty terms are introduced for the $(m - (n+1)(L-\ell+1))$-th order derivatives when defining the bilinear form, as discussed in Section \ref{subsec:IP-form}.

\section{Interior penalty method for the nonconforming elements} 
\label{sec:IP-nonconforming}

In this section, we derive the interior penalty nonconforming methods
for the $2m$-th order partial differential equations
\eqref{equ:m-harmonic}. 

\subsection{Interior penalty nonconforming methods} \label{subsec:IP-form}
Based on the discussion in the previous section, the nonconforming finite element spaces do not generally satisfy the weak continuity or weak zero-boundary condition when $L \geq 1$. To address this issue, we introduce an interior penalty as a remedy. Let $V_h = V_{h0}^{(m,n)}$ denote the nonconforming approximation of $H_0^m(\Omega)$. For any $w, v \in V_h + H_0^m(\Omega)$, we define the following bilinear form:  
\begin{equation}\label{equ:IP-bilinear} 
\begin{aligned}
a_h(w, v) &:= \sum_{|\balpha|=m}(b_{\balpha} \partial_h^\balpha
w, \partial_h^\balpha v) \\ 
&\quad + \eta \sum_{\ell=1}^L \sum_{F\in \mathcal{F}_h} h_F^{1-
2(n+1)(L-\ell+1)} \int_F \sum_{|\bbeta| = m-(n+1)(L-\ell+1)}
\llbracket \partial_h^{\bbeta} w\rrbracket \cdot \llbracket \partial_h^{\bbeta} v\rrbracket,
\end{aligned}
\end{equation} 
where $\eta = \mathcal{O}(1)$ is a given positive constant.  
The interior penalty nonconforming finite element methods for problem \eqref{equ:m-harmonic} is then given by: find $u_h \in V_h$ such that  
\begin{equation} \label{equ:IP-mn}
a_h(u_h, v_h) = (f,v_h) \qquad \forall v_h \in V_h.
\end{equation}  

Furthermore, we define the seim-norm  
\begin{equation} \label{equ:norm-mn}
\|v\|_h^2 := |v|_{m,h}^2 + \sum_{\ell=1}^L \sum_{F\in
\mathcal{F}_h} h_F^{1-2(n+1)(L-\ell+1)} \sum_{|\bbeta| =
m - (n+1)(L-\ell+1)}\|\llbracket \partial^{\bbeta}v \rrbracket \|_{0,F}^2 \qquad
  \forall v \in V_h + H_0^m(\Omega),
\end{equation}  
which can be proven to be a norm on $V_h + H_0^m(\Omega)$ by Theorem \ref{thm:Poincare} presented in the next subsection.

\begin{remark}[extension of the Morley-Wang-Xu nonconforming FEM]
For the case in which $1 \leq m \leq n$, we have $L = \lfloor
  \frac{m}{n+1} \rfloor = 0$. Then, $a_h(\cdot, \cdot)$ and
  $\|\cdot\|_h$ are exactly the bilinear form and norm for the
  nonconforming finite element methods, respectively.
\end{remark}

\begin{remark}[nonstandard interior penalty formulation] 
In the bilinear form \eqref{equ:IP-bilinear}, all components are semi-positive definite, and no indefinite cross terms are present. This distinguishes the formulation from traditional IPDG methods. The essential reason lies in the fact that certain weak continuity properties of the nonconforming elements can control the consistency error. Consequently, by combining nonconforming elements with interior penalty techniques, the proposed method not only simplifies the bilinear form but also offers another significant advantage: the parameter $\eta$ only needs to be a positive constant of $\mathcal{O}(1)$ (without requiring it to be sufficiently large).
\end{remark}

\subsection{Well-posedness}
In this subsection, we study the well-posedness of \eqref{equ:IP-mn}.

\begin{lemma}[jump estimate] \label{lem:jump-estimate}
For any $i < m$, it holds that 
\begin{equation} \label{equ:jump-estimate} 
\sum_{F \in \mathcal{F}_h} \sum_{|\balpha|=i} \max_{y \in F}
|[\partial_h^{\balpha}v_h](y)| \lesssim h^{m-i-n/2} \|v_h\|_h \qquad
\forall v_h \in V_h.
\end{equation}
\end{lemma}
\begin{proof} 
The proof proceeds by considering the following two cases:  
\begin{enumerate}
\item Case 1: $i = m-(n+1)(L-\ell+1)$ for some $1 \leq \ell \leq L$.  
In this case, note that the corresponding jump terms are already included in the definition of $\|\cdot\|_h$. Consequently, we have:
\begin{equation*}
\sum_{F \in \mathcal{F}_h} \sum_{|\balpha|=i} \max_{y \in F} |[\partial^\balpha v_h](y)|^2 
\lesssim \sum_{F \in \mathcal{F}_h} h_F^{1-n} \sum_{|\balpha|=i} \|[\partial^\balpha v_h]\|_{0,F}^2 
\lesssim h^{2(m-i)-n} \|v_h\|_h^2.
\end{equation*}
Here, we note that $|[\partial_h^\balpha v_h(y)]| = |\llbracket \partial_h^\balpha v_h(y)\rrbracket|$ for any $y \in F$.  

\item Case 2: $i \neq m-(n+1)(L-\ell+1)$ for any $1 \leq \ell \leq L$.  
In this case, by leveraging Lemma \ref{lem:k-weak-continuity} (property of $p$-weak continuity and $p$-weak zero-boundary condition), we can estimate the jump terms as follows:  
\begin{equation*}
\sum_{F \in \mathcal{F}_h} \sum_{|\balpha|=i} \max_{y \in F} |[\partial^\balpha v_h](y)|^2 
\lesssim \sum_{F \in \mathcal{F}_h} h_F^2 \sum_{|\bbeta|=i+1} \max_{y \in F} |[\partial^\bbeta v_h](y)|^2.
\end{equation*}
By iterating this argument and using the result established in Case 1, we deduce the desired estimate for this case.
\end{enumerate}
Combining the results from the two cases completes the proof.
\end{proof}

\begin{lemma}[$H^1$ weak approximation] \label{lem:H1-weak-approximation}
For any $v_h \in V_h$ and $|\balpha| < m$, there exists a piecewise polynomial $v_{\balpha} \in H_0^1(\Omega)$ such that 
\begin{equation} \label{equ:H1-weak-approx}
|\partial_h^{\balpha}v_h - v_{\balpha}|_{j,h} \lesssim h^{m-|\balpha|-j}
\|v_h\|_h, \qquad 0 \leq j \leq m - |\balpha|.
\end{equation}
\end{lemma}

\begin{proof}
The proof follows a similar argument to those in \cite{wang2001necessity} and \cite[Lemma 3.1]{wang2013minimal}. It is presented here for completeness.  

For any $v_h \in V_h$ and $|\balpha| < m$, note that $\partial_h^{\balpha} v_h \in W_h^{m-|\balpha|}$, where 
\[
W_h^{m-|\balpha|} := \{ w \in L^2(\Omega) \,:\, w|_T \in
\mathcal{P}_{m-|\balpha|}(T), \, \forall T \in \mathcal{T}_h \}.
\] 
Let $S_h^{m-|\balpha|}$ be the $\mathcal{P}_{m-|\balpha|}$-Lagrangian finite element space on $\mathcal{T}_h$ (cf. \cite{ciarlet1978finite, brenner2007mathematical}), and $\Xi_T^{m-|\balpha|}$ the set of nodal points on $T$. We define the operator $\Pi_h^{p,m-|\balpha|}: W_h^{m-|\balpha|} \mapsto S_h^{m-|\balpha|}$ as follows: for each $x \in \Xi_T^{m-|\balpha|}$,  
\begin{equation} \label{equ:Pl-Lagrangian} 
\Pi_h^{p,m-|\balpha|} v(x) :=
\frac{1}{N_h(x)} \sum_{T' \in \mathcal{T}_h(x)} v|_{T'}(x), 
\end{equation}
where $\mathcal{T}_h(x) = \{ T' \in \mathcal{T}_h \,:\, x \in T' \}$, and $N_h(x)$ denotes the cardinality of $\mathcal{T}_h(x)$. 

Further, let $S_{h0}^{m-|\balpha|} = S_h^{m-|\balpha|} \cap H_0^1(\Omega)$. The operator $\Pi_{h0}^{p,m-|\balpha|}: W_h^{m-|\balpha|} \mapsto S_{h0}^{m-|\balpha|}$ is then defined as follows: for each $x \in \Xi_T^{m-|\balpha|}$,  
\begin{equation} \label{equ:Pl0-Lagrangian}
\Pi_{h0}^{p,m-|\balpha|} v(x) := 
\begin{cases} 
0, & x \in \partial \Omega, \\
\Pi_h^{p,m-|\balpha|} v(x), & \text{otherwise}. 
\end{cases}
\end{equation}
Define $v_{\balpha} := \Pi_{h0}^{p,m-|\balpha|} v_h$. Using the standard scaling argument (cf. \cite{brenner2005c0, brenner2009posteriori}) and Lemma \ref{lem:jump-estimate} (jump estimate), we have:
\[
\begin{aligned}
\|\partial_h^{\balpha} v_h - v_{\balpha}\|_0^2 
& \lesssim \sum_{T \in \mathcal{T}_h} h_T^n \sum_{x \in \Xi_T^{m-|\balpha|}} |\partial_h^{\balpha} v_h(x) - v_{\balpha}(x)|^2 \\ 
& \lesssim \sum_{F \in \mathcal{F}_h} h_F^n \max_{y \in F} |[\partial_h^{\balpha} v_h](y)|^2  \lesssim h^{2(m-|\balpha|)} \|v_h\|_h^2.
\end{aligned}
\]
The remainder of the proof follows by applying the inverse inequality to establish the desired estimates for higher-order seim-norms. 
\end{proof}

\begin{theorem}[discrete Poincar\'{e} inequality] \label{thm:Poincare}
The following discrete Poincar\'{e} inequality holds: 
\begin{equation} \label{equ:Poincare}
\|v\|_{m,h} \lesssim \|v\|_{h} \qquad \forall v \in V_h+H_0^m(\Omega).
\end{equation}
\end{theorem}
\begin{proof} The case in which $v\in H_0^m(\Omega)$ is obvious. Next,
we consider the case in which $v \in V_h$. For any $|\balpha| = p <
m$, from Lemma \ref{lem:H1-weak-approximation} ($H^1$ weak approximation), we obtain 
$$ 
\|\partial^{\balpha}_h v\|_{0}^2 \lesssim \|\partial_h^{\balpha} v
- v_{\balpha}\|_{0}^2 + \|v_{\balpha}\|_{0}^2 \lesssim
\|v\|_{h}^2 + |v_{\balpha}|_{1}^2 \lesssim \|v\|_h^2 + |v|_{p+1,h}^2.
$$ 
Consequently, we have  
$$ 
|v|_{p,h}^2 \lesssim \|v\|_h^2 + |v|_{p+1,h}^2 \qquad 0 \leq p <m,
$$ 
which leads to \eqref{equ:Poincare} by induction.
\end{proof}

Theorem \ref{thm:Poincare} demonstrates that $\|\cdot\|_h$ defined in \eqref{equ:norm-mn} is indeed a norm. We now establish the well-posedness of the problem in \eqref{equ:IP-mn}, as follows.

\begin{theorem}[well-posedness] \label{thm:well-posedness}
Let $\eta = \mathcal{O}(1)$ be a given positive constant. We have the following results:
\begin{align}
a_h(v,w) & \leq \max\{1,\eta\} \|v\|_h \|w\|_h \quad \forall v,w \in V_h + H_0^{m}(\Omega), \label{equ:boundedness} \\
\min\{1,\eta\} \|v\|_h^2 & \leq a_h(v,v) \quad \forall v \in V_h + H_0^{m}(\Omega). \label{equ:coercive}
\end{align}
As a result, there exists a unique solution to the problem in \eqref{equ:IP-mn}.
\end{theorem}

\section{Error estimates} \label{sec:convergence}

First, we demonstrate the approximation property of $V_h^{(m,n)}$ and $V_{h0}^{(m,n)}$, which follows directly from Lemma \ref{lem:interpolation} (local interpolation error).

\begin{theorem}[interpolation error] \label{thm:approximation}
For $s \in [0,1]$, it holds that
\begin{equation}
\|v - \Pi_h^{(m,n)}v\|_h \lesssim h^s |v|_{m+s} \qquad \forall v \in H^{m+s}(\Omega) \cap H_0^{m}(\Omega).
\end{equation}
\end{theorem}

\begin{proof}
For any $T \in \mathcal{T}_h$, by Lemma \ref{lem:interpolation}, we have
\[
|v - \Pi_T^{(m,n)}v|_{m,T} \lesssim h_T^s |v|_{m+s,T}.
\]
By the trace inequality, for $k = m - (n+1)(L - \ell + 1)$, where $1 \leq \ell \leq L$, we obtain
\[
\begin{aligned}
h_T^{1-2m+2k} \|v - \Pi_T^{(m,n)}v\|_{k, \partial T}^2 &\lesssim h_T^{1-2m+2k} \left( h_T |v - \Pi_T^{(m,n)}v|_{k+1,T}^2 + h_T^{-1} |v - \Pi_T^{(m,n)}v|_{k,T}^2 \right) \\
&\lesssim h_T^{1-2m+2k} \left( h_T^{2m + 2s - 2k - 1} |v|_{m+s,T}^2 \right) = h_T^{2s} |v|_{m+s,T}^2.
\end{aligned}
\]
This completes the proof.
\end{proof}

Based on Strang's Lemma, we have
\begin{equation} \label{equ:Strang}
\|u - u_h\|_h \lesssim \inf_{v_h \in V_h} \|u - v_h\|_h + \sup_{v_h \in V_h} \frac{|a_h(u,v_h) - (f, v_h)|}{\|v_h\|_h}.
\end{equation}
The first term on the right-hand side represents the approximation error, which can be estimated using Theorem \ref{thm:approximation}.
Next, we consider the estimate for the consistent error term.

Given $|\balpha| = m$, it can be written as $\balpha = \sum_{i=1}^m
\be_{j_{\balpha, i}}$, where $\be_i ~(i=1,\cdots, n)$ are the unit
vectors in $\mathbb{R}^n$. We also set $ \balpha_{(k)} = \sum_{i=1}^k
\be_{j_{\balpha,i}}$. 

\subsection{Error Estimate under the Extra Regularity Assumption}
In this subsection, we present the error estimate under the additional regularity assumption, namely $u \in H^r(\Omega)$, where $r = \max\{m+1, 2m-1\}$.

\begin{lemma}[consistency error I] \label{lem:nonconforming-regularity}
Let $r = \max\{m+1, 2m-1\}$. If $u \in H^r(\Omega)$ and $f \in
L^2(\Omega)$, then
\begin{equation} \label{equ:nonconforming-regularity}
\sup_{v_h \in V_h} \frac{|a_h(u, v_h) - (f, v_h)|}{\|v_h\|_{h}} \lesssim \sum_{k=1}^{r-m} h^k |u|_{m+k} + h^m \|f\|_0.
\end{equation}
\end{lemma}

\begin{proof}
Note that $u\in H_0^m(\Omega)$,  we first have 
$$ 
\begin{aligned}
a_h(u, v_h) - (f, v_h) &= \sum_{T\in \mathcal{T}_h} \int_T \left(
\sum_{|\balpha| = m} \partial^\balpha u \partial^\balpha v_h \right) -
(f, v_h) \\
& = \sum_{|\balpha| = m} \sum_{T\in \mathcal{T}_h} \int_T
\partial^\balpha u \partial^\balpha v_h - (-1)^m(\partial^{2\balpha}u) v_h :=
E_1 + E_2 + E_3,
\end{aligned}
$$ 
where 
$$ 
\begin{aligned}
E_1 &:= \sum_{|\balpha| = m} \sum_{T\in \mathcal{T}_h} \int_T 
\partial^\balpha u \partial^\balpha v_h + \partial^{\balpha+\balpha_{(1)}}
u \partial^{\balpha - \balpha_{(1)}} v_h, \\
E_2 &:= \sum_{k=1}^{m-2} (-1)^k \sum_{|\balpha| = m} \sum_{T\in
  \mathcal{T}_h} \int_T \partial^{\balpha + \balpha_{(k)}} u
  \partial^{\balpha - \balpha_{(k)}} v_h + \partial^{\balpha +
    \balpha_{(k+1)}}u \partial^{\balpha - \balpha_{(k+1)}} v_h, \\ 
E_3 &:= (-1)^{m-1}\sum_{|\balpha|=m} \sum_{T \in \mathcal{T}_h} \int_T
\partial^{2\balpha - \be_{j_{\balpha,m}}} u \partial^{\be_{j_{\balpha,m}}}
v_h + (\partial^{2\balpha} u) v_h.
\end{aligned}
$$ 
By Lemma \ref{lem:face-average} and Green's formula, we have   
$$ 
\begin{aligned}
E_1 &= \sum_{|\balpha|=m} \sum_{T \in \mathcal{T}_h} \int_{\partial T}
\partial^\balpha u \partial^{\balpha - \be_{j_{\balpha,1}}} v_h
\nu_{j_{\balpha},1} \\
&= \sum_{|\balpha|=m} \sum_{T\in \mathcal{T}_h} \sum_{F\subset \partial
T} \int_F \left(\partial^\balpha u - P_F^0 \partial^\balpha u \right)
(\partial_h^{\balpha - \be_{j_{\balpha,1}}}v_h - P_F^0 \partial_h^{\balpha -
 \be_{j_{\balpha,1}}}v_h)\nu_{j_{\balpha},1},
\end{aligned}
$$ 
where $P_F^0: L^2(F) \mapsto \mathcal{P}_0(F)$ is the orthogonal
projection, $\nu = (\nu_1, \cdots, \nu_n)$ is the unit outer normal to
$\partial T$. Using the Schwarz inequality and the interpolation theory, we
obtain 
\begin{equation} \label{equ:regularity-E1}
|E_1| \lesssim h|u|_{m+1} |v_h|_{m,h} \leq h|u|_{m+1} \|v_h\|_h.
\end{equation} 

When $m>1$, let $v_\bbeta \in H_0^1(\Omega)$ be the piecewise
polynomial as in Lemma \ref{lem:H1-weak-approximation} ($H^1$ weak approximation). Then, 
Green's formula leads to 
$$ 
\begin{aligned}
E_2 &= \sum_{k=1}^{m-2} (-1)^k \sum_{|\balpha| = m} \sum_{T\in
  \mathcal{T}_h} \int_T \partial^{\balpha + \balpha_{(k)}} u
  \partial^{\be_{j_{\balpha},k+1}} (\partial^{\balpha-\balpha_{(k+1)}}v_h
   - v_{\balpha-\balpha_{(k+1)}}) \\
    &+ \sum_{k=1}^{m-2} (-1)^k \sum_{|\balpha| = m} \sum_{T\in
  \mathcal{T}_h} \int_T \partial^{\balpha + \balpha_{(k+1)}}u  (
      \partial^{\balpha - \balpha_{(k+1)}} v_h -
      v_{\balpha-\balpha_{(k+1)}} ), \\ 
\end{aligned}
$$ 
which implies 
\begin{equation} \label{equ:regularity-E2}
|E_2| \lesssim \sum_{k=1}^{m-2} h^k|u|_{m+k} \|v_h\|_{h} +
h^{k+1}|u|_{m+k+1} \|v_h\|_{h}.
\end{equation}
Finally, we have 
$$
\begin{aligned}
E_3 &= (-1)^{m-1}\sum_{|\balpha|=m} \sum_{T \in \mathcal{T}_h} \int_T
\partial^{2\balpha - \be_{j_{\balpha,m}}} u \partial^{\be_{j_{\balpha,m}}}
(v_h-v_0) + (\partial^{2\balpha} u) (v_h-v_0),
\end{aligned}
$$
which gives 
\begin{equation} \label{equ:regularity-E3} 
|E_3| \lesssim h^{m-1}|u|_{2m-1}\|v_h\|_{h} + h^m \|f\|_0 \|v_h\|_{h}.
\end{equation} 
By the estimates \eqref{equ:regularity-E1}, \eqref{equ:regularity-E2},
and \eqref{equ:regularity-E3}, we obtain the desired estimate
\eqref{equ:nonconforming-regularity}.
\end{proof}

We have the following theorem.
\begin{theorem}[error estimate I] \label{thm:error-regularity}
Let $r = \max\{m+1, 2m-1\}$. If $u \in H^r(\Omega) \cap H_0^m(\Omega)$
and $f \in L^2(\Omega)$, then 
\begin{equation} \label{equ:error-regularity}
\|u - u_h\|_{h} \lesssim \sum_{k=1}^{r-m} h^k |u|_{m+k} + h^m \|f\|_0
= \mathcal{O}(h).
\end{equation}
\end{theorem}

\begin{remark}[consistency with existing results]
For the cases where $m \leq n$, we have $\|\cdot\|_h =
|\cdot|_{m,h}$. Thus, Theorem \ref{thm:error-regularity} is consistent
with the error estimate for the Morley element (cf. \cite{shi1990error}) and
the main theorem in \cite[Theorem 3.4]{wang2013minimal}. 
\end{remark}

\subsection{Error estimate by conforming relatives}

In this subsection, we will provide error estimates for the interior penalty nonconforming finite element methods under the minimum regularity requirement, by constructing enriching operators to the $H^m$ conforming finite element space. In fact, this technique has been widely used in the analysis of nonconforming elements: see \cite{scott1990finite, brenner2003poincare} for
the case in which $m=1$, \cite{brenner2005c0, li2014new} for $m=2$ in 2D, \cite{zhang2009family} for $m=2$ in 3D, and \cite{hu2014new, hu2016canonical} for arbitrary $m \geq 1$ in 2D. 

However, the existence of conforming finite element spaces on simplex meshes in arbitrary dimension or arbitrary order has long been a missing piece in the rigorous analysis of this technique. Only recently, Hu, Lin, and Wu in \cite{hu2023construction} provided the final piece, completing this technique. Specifically, \cite{hu2023construction} gives an $H^m$ conforming finite element space on $\mathcal{T}_h$, which we denote as $V_h^c \subset H_0^m(\Omega)$. The shape function space is denoted by $P_T^c$, and the global degrees of freedom are denoted by $D^c$.

\begin{remark}
The definition of conforming relatives used here differs slightly from those in Brenner \cite{brenner1996two} and Brenner and Sung \cite{brenner2005c0}, where the conditions $P_T^{(m,n)} \subset P_T^c$ and $D_T^{(m,n)} \subset D^c|_T$ are imposed. First, since $P_T^{(m,n)} = \mathcal{P}_m(T)$ is minimal, the inclusion $P_T^{(m,n)} \subset P_T^c$ always holds. Furthermore, as demonstrated below, the analysis in this subsection does not rely on the condition $D_T^{(m,n)} \subset D^c|_T$.
\end{remark}

A crucial component in the analysis of interior penalty nonconforming methods is the construction of the operator $\Pi_h^c: V_h \to V_h^c$, often referred to as the enriching operator in \cite{brenner2005c0, gudi2011interior}. The operator $\Pi_h^c$ is defined as follows: For any $d \in D^c$,  
\begin{equation} \label{eq:Pic}  
d(\Pi_h^c v_h) := \left\{   
\begin{aligned}  
& 0 \qquad\qquad\qquad\qquad~ \text{if } d(w) = 0 \text{ for all } w \in H_0^m(\Omega), \\   
& \frac{1}{|\mathcal{T}_{d}|} \sum_{T \in \mathcal{T}_{d}} d(v_h|_T) \qquad \text{otherwise}.  
\end{aligned}  
\right.  
\end{equation}  
Here, $\mathcal{T}_{d} \subset \mathcal{T}_h$ represents the set of simplices sharing the degree of freedom $d$, and $|\mathcal{T}_d|$ denotes the cardinality of this set. Additionally, the highest order of derivatives in $D^c$ is denoted by $M$.  

\begin{lemma}[enriching operator] \label{lem:conforming-relatives}
For $\Pi_h^c$ defined in \eqref{eq:Pic}, the following inequality holds:
\begin{equation} \label{equ:relative-enriching}
\sum_{j=0}^{m-1} h^{2(j-m)} |v_h - \Pi_h^c v_h|_{j,h}^2 + |\Pi_h^c
v_h|_{m,h}^2 \lesssim \|v_h\|_h^2, \qquad \forall v_h \in V_h. 
\end{equation}
\end{lemma}

\begin{proof}
The proof is based on arguments similar to those in \cite{brenner1996two,
brenner2005c0, brenner2009posteriori, gudi2011interior}. The main distinction lies in the need to estimate certain jump terms. Below, we provide only a brief outline of the proof. 

First, using the standard scaling argument, we have 
\begin{equation} \label{equ:enriching-scaling} 
\sum_{T\in \mathcal{T}_h} \|v_h - \Pi_h^c v_h\|_{0,T}^2 \lesssim
\sum_{i=0}^M \sum_{F\in \mathcal{F}_h }h_F^{2i+n}
\sum_{|\balpha|=i}\max_{y\in F}|[\partial_h^{\balpha}v_h](y)|^2. 
\end{equation} 
We estimate the right-hand side of \eqref{equ:enriching-scaling} in two cases:

\begin{enumerate}
\item $i \geq m$: Using the standard scaling argument and the inverse inequality, we have
$$ 
\sum_{F\in \mathcal{F}_h} h_F^{2i+n}\sum_{|\balpha|=i} \max_{y\in
  F}|[\partial^\balpha v_h](y)|^2 \lesssim h^{2m} |v_h|_{m}^2 \leq
  h^{2m} \|v_h\|_{h}^2.
$$ 

\item $i < m$: By applying Lemma \ref{lem:jump-estimate} (jump estimate), it follows that
$$ 
\sum_{F\in \mathcal{F}_h} h_F^{2i+n}\sum_{|\balpha|=i}\max_{y\in
F}|[\partial^{\balpha}_h v_h](y)|^2 \lesssim 
h^{2i+n} h^{2m-2i-n} \|v_h\|_h^2 = h^{2m} \|v_h\|_h^2.
$$ 
\end{enumerate}

Combining the results of both cases, we conclude that  
$$ 
\|v_h - \Pi_h^c v_h\|_{0}^2 \lesssim h^{2m}\|v_h\|_h^2.
$$ 
The remaining terms can be bounded using the inverse inequality, completing the proof.
\end{proof}

Let $\mathcal{P}_0(\mathcal{T}_h)$ denote the space of piecewise constant functions over the mesh $\mathcal{T}_h$. To establish the quasi-optimal error estimate under minimal regularity assumption, we first define the piecewise constant projection $P_h^0: L^2(\Omega) \to \mathcal{P}_0(\mathcal{T}_h)$ as follows:  
\begin{equation} \label{equ:constant-proj}
P_h^0 v|_T := \frac{1}{|T|} \int_T v, \qquad \forall T \in \mathcal{T}_h.
\end{equation}  
Additionally, we introduce the average operator over $\omega_F$ as:  
\begin{equation} \label{equ:average-omega}
P_{\omega_F}^0v := \frac{1}{|\omega_F|} \int_{\omega_F} v.
\end{equation}

\begin{lemma}[consistency error II] \label{lem:nonconforming-no-regularity}
If $f \in L^2(\Omega)$, then we have 
\begin{equation} \label{equ:nonconforming-no-regularity}
\begin{aligned}
\sup_{v_h \in V_h} \frac{|a_h(u, v_h) - \langle f, v_h \rangle
|}{\|v_h\|_{h}} &\lesssim \inf_{w_h\in V_h} \|u-w_h\|_{h} +
h^m \|f\|_0 \\
&+ \sum_{|\balpha|=m} \left(\|\partial^\balpha u - P_h^0
    \partial^\balpha u\|_0 + \sum_{F\in \mathcal{F}_h}
    \|\partial^\balpha u - P_{\omega_F}^0 \partial^\balpha
    u\|_{0,\omega_F} \right). 
\end{aligned}
\end{equation}
\end{lemma}

\begin{proof}
For any $w_h \in V_h$,
$$ 
\begin{aligned}
a_h(u, v_h) - (f,v_h) &= a_h(u, v_h - \Pi_h^c v_h) - (f, v_h - \Pi_h^c
v_h) \\
&= a_h(u - w_h, v_h - \Pi_h^c v_h) + a_h(w_h, v_h - \Pi_h^c v_h) - (f,
v_h - \Pi_h^c v_h)
\end{aligned}
$$ 
For the first and third terms, we have 
\begin{equation} \label{equ:no-regularity1} 
\begin{aligned}
|a_h(u-w_h, v_h - \Pi_h^c v_h)| &\lesssim \|u-w_h\|_{h} \|v_h -
\Pi_h^c v_h\|_{h} \lesssim \|u-w_h\|_{h} \|v_h\|_{h}, \\
|(f, v_h - \Pi_h^c v_h)| & \lesssim \|f\|_{0} \|v_h - \Pi_h^c
v_h\|_0 \lesssim h^m\|f\|_{0} \|v_h\|_{h}.
\end{aligned}
\end{equation}
Next, we estimate the second term. First,
$$ 
\begin{aligned}
a_h(w_h, v_h - \Pi_h^c v_h) &= \sum_{|\balpha|=m}\sum_{T\in
  \mathcal{T}_h} \int_T \partial^\balpha w_h \partial^\balpha (v_h -
      \Pi_h^c v_h) \\ 
&~~ + \eta \sum_{l=1}^L\sum_{F\in \mathcal{F}_{h}}
h_F^{1-2(n+1)(L-l+1)} \int_F \sum_{|\bbeta|=m-(n+1)(L-l+1)}
\llbracket \partial_h^{\bbeta}w_h \rrbracket \cdot \llbracket \partial_h^{\bbeta}v_h\rrbracket \\
&:= E_1 + E_2 + E_3, 
\end{aligned}
$$ 
where
$$ 
\begin{aligned}
E_1 &:= \sum_{|\balpha| = m} \sum_{T\in \mathcal{T}_h} \int_T 
\partial^\balpha w_h \partial^\balpha (v_h-\Pi_h^cv_h) +
\partial^{\balpha+\be_{j_{\balpha},1}} w_h \partial^{\balpha -
  \be_{j_{\balpha},1}}
(v_h - \Pi_h^c v_h), \\
E_2 &:= -\sum_{|\balpha| = m} \sum_{T\in
  \mathcal{T}_h} \int_T \partial^{\balpha + \be_{j_\balpha,1}}w_h
  \partial^{\balpha - \be_{j_\balpha,1}} (v_h-\Pi_h^c v_h), \\
E_3 &:= \eta \sum_{l=1}^L \sum_{F\in
  \mathcal{F}_h} h_F^{1-2(n+1)(L-l+1)} \int_F
  \sum_{|\bbeta|=m-(n+1)(L-l+1)}
 \llbracket \partial_h^{\bbeta} w_h \rrbracket \cdot \llbracket \partial_h^{\bbeta} v_h \rrbracket.
\end{aligned}
$$ 
First we have 
$$ 
E_3 = \eta \sum_{l=1}^L \sum_{F\in
\mathcal{F}_h} h_F^{1-2(n+1)(L-l+1)} \int_F
\sum_{|\bbeta|=m-(n+1)(L-l+1)} \llbracket \partial_h^{\bbeta} (u-w_h) \rrbracket \cdot
\llbracket \partial_h^{\bbeta} v_h \rrbracket.
$$ 
which means that 
\begin{equation} \label{equ:no-regularity-E3}
|E_3| \lesssim \|u-w_h\|_h \|v_h\|_h.
\end{equation}

By Lemma \ref{lem:face-average} and Green's formula, we have   
$$ 
\begin{aligned}
E_1 &= \sum_{|\balpha|=m} \sum_{T \in \mathcal{T}_h} \int_{\partial T}
\partial_h^\balpha w_h \partial_h^{\balpha - \be_{j_\balpha,1}} (v_h -
\Pi_h^c v_h) \nu_{j_{\balpha,1}}\\
&= \sum_{|\balpha|=m} \sum_{F\in \mathcal{F}_h} \int_F
\{\partial_h^\balpha w_h\}\llbracket \partial_h^{\balpha - \be_{j_\balpha,1}} (v_h -
    \Pi_h^c v_h) \rrbracket_{j_{\balpha,1}} \\
&+ \sum_{|\balpha|=m} \sum_{F\in \mathcal{F}_h^i} \int_F
\llbracket\partial_h^\balpha w_h\rrbracket_{j_{\balpha,1}} 
\{\partial_h^{\balpha-\be_{j_\balpha,1}} (v_h - \Pi_h^c v_h)\} \qquad
(\text{Eq. (3.3) in \cite{arnold2002unified}})\\
&= \sum_{|\balpha|=m} \sum_{F\in \mathcal{F}_h} \int_F
\{\partial_h^\balpha w_h - P_h^0 \partial^\balpha u\}\llbracket
\partial_h^{\balpha - \be_{j_\balpha,1}} (v_h - \Pi_h^c
    v_h)\rrbracket_{j_{\balpha,1}} \\
&+ \sum_{|\balpha|=m} \sum_{F\in \mathcal{F}_h^i} \int_F \llbracket \partial_h^\balpha
w_h - P_{\omega_F}^0 \partial^\balpha u\rrbracket_{j_{\balpha,1}}
\{\partial_h^{\balpha-\be_{j_\balpha,1}} (v_h - \Pi_h^c v_h)\}.
\end{aligned}
$$ 
Therefore, it follows from the trace inequality and inverse inequality
that 
\begin{equation}\label{equ:no-regularity-E1}
\begin{aligned}
|E_1| &\lesssim \sum_{|\balpha|=m} \sum_{F\in \mathcal{F}_h} h_F^{-1}
\|\partial_h^\balpha w_h - P_h^0 \partial^\balpha u\|_{0,\omega_F} |v_h
- \Pi_h^c v_h|_{m-1,h} \\
&+ \sum_{|\balpha|=m} \sum_{F\in \mathcal{F}_h}h_F^{-1} 
\|\partial_h^\balpha w_h - P_{\omega_F}^0 \partial^\balpha
u\|_{0,\omega_F} |v_h - \Pi_h^c v_h|_{m-1,h} \\
&\lesssim \sum_{|\balpha|=m} \left(\|\partial_h^\balpha w_h - P_h^0
\partial^\balpha u\|_0 + \sum_{F\in \mathcal{F}_h} \|\partial_h^\balpha w_h
- P_{\omega_F}^0 \partial^\balpha u\|_{0,\omega_F} 
\right) \|v_h\|_{h}.
\end{aligned}
\end{equation}

For the estimate of $E_2$, we obtain
$$ 
\begin{aligned}
E_2 &= -\sum_{|\balpha| = m} \sum_{T\in \mathcal{T}_h} \int_T
\partial^{\be_{j_\alpha,1}}(\partial^{\balpha}w_h - P_h^0
  \partial^{\balpha} u) \partial^{\balpha - \be_{j_\alpha,1}}(v_h -
    \Pi_h^c v_h), 
\end{aligned}
$$
which gives 
\begin{equation} \label{equ:no-regularity-E2} 
\begin{aligned}
|E_2| &\lesssim \sum_{|\balpha|=m} h^{-1}
\|\partial_h^\balpha w_h - P_h^0 \partial^\balpha u\|_0 |v_h - \Pi_h^c
v_h|_{m-1,h} \\
&\lesssim \sum_{|\balpha|=m} \|\partial_h^\balpha w_h -
P_h^0\partial^\balpha u\|_0 \|v_h\|_{h}. 
\end{aligned}
\end{equation} 
We therefore complete the proof by \eqref{equ:no-regularity1}, \eqref{equ:no-regularity-E3},
\eqref{equ:no-regularity-E1}, \eqref{equ:no-regularity-E2}, and the
triangle inequality.
\end{proof}

Using Lemma \ref{lem:nonconforming-no-regularity} (consistency error) and interpolation property, we immediately obtain the following theorem:
\begin{theorem}[error estimate II] \label{thm:error-no-regularity}
If $f \in L^2(\Omega)$ and $u \in H^{m+t}(\Omega)$, then the following holds:
\begin{equation}\label{equ:error-no-regularity}
\|u - u_h\|_{h} \lesssim h^{s}|u|_{m+s} + h^m \|f\|_0,
\end{equation}
where $s = \min\{1, t\}$.
\end{theorem}

\begin{remark}
We note that the minimal regularity is required in Lemma
\ref{lem:nonconforming-no-regularity} and Theorem
\ref{thm:error-no-regularity}. Similar technique can be found in
\cite{mao2010error,hu2014new,hu2016canonical,wu2017nonconforming}.
\end{remark}


\section{Numerical Tests} \label{sec:numerical}
In this section, we present several numerical results to support the theoretical findings established in Section \ref{sec:convergence}. Note that since it suffices to have $\eta = \mathcal{O}(1)$, we simply set $\eta = 1$.


\subsection{Smooth solution when $m=3, n=2$}
In this example, we set $f=0$, resulting in the exact solution $u = \exp(\pi y) \sin(\pi x)$ on $\Omega = (0,1)^2$, with nonhomogeneous boundary conditions. By solving \eqref{equ:IP-mn} for various values of $h$, we compute the errors and orders of convergence in borken $H^k$ norms for $k=0,1,2,3$, and summarize the results in Table \ref{tab:2D-example1}. 

The table demonstrates that the computed solution converges linearly to the exact solution in the $H^3$ norm, which aligns with the theoretical predictions in Theorem \ref{thm:error-regularity} (error estimate I) and Theorem \ref{thm:error-no-regularity} (error estimate II).

\begin{table}[!htbp]
\caption{Example 1: Errors and observed convergence orders.}
\centering
{\small{
\begin{tabular}{@{}c|cc|cc|cc|cc@{}}
   \hline
  $1/h$	&$\|u-u_h\|_0$	& Order	& 
  $|u-u_h|_{1,h}$ & Order & 
  $|u-u_h|_{2,h}$	& Order & $|u-u_h|_{3,h}$ & Order\\ \hline
  8		&2.1388e-2 &--	  &2.8269e-1 &--	  &2.4606e+0 &--   &8.5726e+1	&--  \\
  16	&3.7707e-3 &2.50	&4.4020e-2 &2.68	&5.9908e-1 &2.04 &4.2855e+1	&1.00\\
  32	&9.8025e-4 &1.94	&6.6082e-3 &2.74	&1.4438e-1 &2.05 &2.1369e+1	&1.00\\
  64	&2.7203e-4 &1.85	&1.5666e-3 &2.08	&3.6289e-2 &1.99 &1.0687e+1	&1.00\\
  \hline
\end{tabular}}}
\label{tab:2D-example1}
\end{table}

\begin{figure}[!htbp]
\caption{Uniform grids for Example 1 and Example 2.}
\label{fig:uniform-grids}
\centering 
\subfloat[Example 1: Unit square domain]{\centering 
   \includegraphics[width=0.3\textwidth]{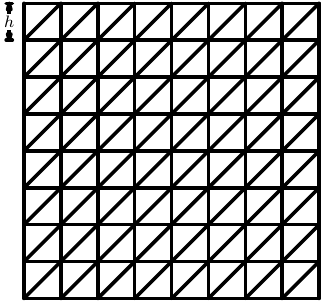} 
   \label{fig:square}
}%
\qquad\qquad  
\subfloat[Example 2: L-shaped domain]{\centering 
   \includegraphics[width=0.3\textwidth]{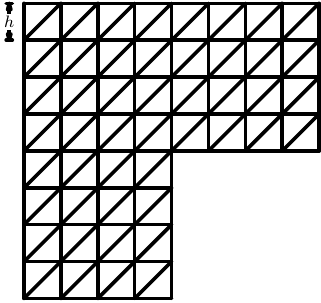} 
   \label{fig:L-shaped}
}
\end{figure}

\subsection{L-shaped domain when $m=3, n=2$}
In this example, we evaluate the method for a case where the solution exhibits partial regularity on a non-convex domain. Specifically, we solve the tri-harmonic equation 
$$ 
(-\Delta)^3 u = 0
$$ 
on the 2D L-shaped domain $\Omega = (-1,1)^2 \setminus [0,1) \times (-1,0]$, as illustrated in Figure \ref{fig:L-shaped}. The exact solution is given by 
$$ 
u = r^{2.5} \sin(2.5\theta),
$$ 
where $(r,\theta)$ are polar coordinates. 

Due to the singularity at
the origin, the solution $u \in H^{3+1/2-\epsilon}(\Omega)$ for any $\epsilon > 0$. As shown in Table \ref{tab:2D-example2}, the method does
converge with the optimal order $h^{1/2}$ in the broken $H^3$
norm.

\begin{table}[!htbp]
\caption{Example 2: Errors and observed convergence orders.}
\centering
{\small{
\begin{tabular}{@{}c|cc|cc|cc|cc@{}}
   \hline
  $1/h$	&$\|u-u_h\|_0$	& Order	& 
  $|u-u_h|_{1,h}$ & Order & 
  $|u-u_h|_{2,h}$	& Order & $|u-u_h|_{3,h}$ & Order\\ \hline
  4		& 1.7045e-3 &--	  &1.3625e-2 &--	  &8.2377e-2 &--   &1.3897e+0	&--  \\
  8		& 5.1684e-4 &1.72	&3.2787e-3 &2.06	&3.1536e-2 &1.39 &1.0045e+0	&0.47\\
  16	& 2.0898e-4 &1.31	&1.1035e-3 &1.57	&1.2527e-2 &1.33 &7.1864e-1	&0.48\\
  32	& 9.0652e-5 &1.21	&4.7740e-4 &1.21	&5.1739e-3 &1.28 &5.1110e-1	&0.49\\
  64	& 4.0534e-5 &1.16	&2.1548e-4 &1.15	&2.1951e-3 &1.24 &3.6240e-1	&0.50\\
  \hline
\end{tabular}}}
\label{tab:2D-example2}
\end{table}

\section{Concluding remarks} \label{sec:concluding}
In this paper, we propose and analyze interior penalty nonconforming finite element methods for $2m$-th order partial differential equations in arbitrary dimensions. The nonconforming finite elements are minimal and consistent with the Morley-Wang-Xu elements when $1 \leq m \leq n$. For cases where $m > n$, interior penalty terms are introduced to compensate for the lack of weak continuity. Owing to the design of the degrees of freedom, the bilinear form incorporates only the jumps of certain derivatives across edges, thereby avoiding the need for excessively large penalty parameters in the corresponding interior penalty scheme.

We further provide two types of error estimates: one assumes additional regularity, while the other relies only on minimal regularity. Due to the simplicity and flexibility of the proposed interior penalty nonconforming methods, the direct discretization of high-order partial differential equations becomes practical and efficient.

\section*{Acknowledgments}
Shuonan Wu was supported in part by the National Natural Science Foundation of China grant No. 12222101 and the Beijing Natural Science Foundation No. 1232007.

\bibliographystyle{amsplain}
\bibliography{IPmn-paper.bbl}

\providecommand{\bysame}{\leavevmode\hbox to3em{\hrulefill}\thinspace}
\providecommand{\MR}{\relax\ifhmode\unskip\space\fi MR }
\providecommand{\MRhref}[2]{%
  \href{http://www.ams.org/mathscinet-getitem?mr=#1}{#2}
}
\providecommand{\href}[2]{#2}
\begin{thebibliography}{10}

\bibitem{adams2003sobolev}
R.~A. Adams and J.~Fournier, \emph{Sobolev spaces}, vol. 140, Academic press,
  2003.

\bibitem{antonietti2020conforming}
Paola~F Antonietti, Gianmarco Manzini, and Marco Verani, \emph{The conforming
  virtual element method for polyharmonic problems}, Computers \& Mathematics
  with Applications \textbf{79} (2020), no.~7, 2021--2034.

\bibitem{arnold2002unified}
Douglas~N Arnold, Franco Brezzi, Bernardo Cockburn, and L~Donatella Marini,
  \emph{{Unified analysis of discontinuous Galerkin methods for elliptic
  problems}}, SIAM Journal on Numerical Analysis \textbf{39} (2002), no.~5,
  1749--1779.

\bibitem{backofen2007nucleation}
Rainer Backofen, Andreas R{\"a}tz, and Axel Voigt, \emph{{Nucleation and growth
  by a phase field crystal (PFC) model}}, Philosophical Magazine Letters
  \textbf{87} (2007), no.~11, 813--820.

\bibitem{barrett2004finite}
John~W Barrett, Stephen Langdon, and Robert N{\"u}rnberg, \emph{Finite element
  approximation of a sixth order nonlinear degenerate parabolic equation},
  Numerische Mathematik \textbf{96} (2004), no.~3, 401--434.

\bibitem{beirao2013basic}
L~Beir{\~a}o~da Veiga, Franco Brezzi, Andrea Cangiani, Gianmarco Manzini,
  L~Donatella Marini, and Alessandro Russo, \emph{Basic principles of virtual
  element methods}, Mathematical Models and Methods in Applied Sciences
  \textbf{23} (2013), no.~01, 199--214.

\bibitem{beirao2014hitchhiker}
L~Beir{\~a}o~da Veiga, Franco Brezzi, Luisa~Donatella Marini, and Alessandro
  Russo, \emph{The hitchhiker's guide to the virtual element method},
  Mathematical models and methods in applied sciences \textbf{24} (2014),
  no.~08, 1541--1573.

\bibitem{bramble1970triangular}
James~H Bramble and Milo{\v{s}} Zl{\'a}mal, \emph{Triangular elements in the
  finite element method}, Mathematics of Computation \textbf{24} (1970),
  no.~112, 809--820.

\bibitem{brenner1996two}
S.~C. Brenner, \emph{{A two-level additive Schwarz preconditioner for
  nonconforming plate elements}}, Numerische Mathematik \textbf{72} (1996),
  no.~4, 419--447.

\bibitem{brenner2009posteriori}
S.~C. Brenner, T.~Gudi, and L.-Y. Sung, \emph{An a posteriori error estimator
  for a quadratic $c^0$-interior penalty method for the biharmonic problem},
  IMA Journal of Numerical Analysis \textbf{30} (2009), no.~3, 777--798.

\bibitem{brenner2007mathematical}
S.~C. Brenner and L.~R. Scott, \emph{The mathematical theory of finite element
  methods}, vol.~15, Springer Science \& Business Media, 2007.

\bibitem{brenner2005c0}
S.~C. Brenner and L.-Y. Sung, \emph{{$C^0$ interior penalty methods for fourth
  order elliptic boundary value problems on polygonal domains}}, Journal of
  Scientific Computing \textbf{22} (2005), no.~1-3, 83--118.

\bibitem{brenner2003poincare}
Susanne~C Brenner, \emph{{Poincar{\'e}--Friedrichs inequalities for piecewise
  $H^1$ functions}}, SIAM Journal on Numerical Analysis \textbf{41} (2003),
  no.~1, 306--324.

\bibitem{chen2022conforming}
Chunyu Chen, Xuehai Huang, and Huayi Wei, \emph{Conforming virtual elements in
  arbitrary dimension}, SIAM Journal on Numerical Analysis \textbf{60} (2022),
  no.~6, 3099--3123.

\bibitem{chen2022c0}
Huangxin Chen, Jingzhi Li, and Weifeng Qiu, \emph{{A $C^0$ interior penalty
  method for $m$th-Laplace equation}}, ESAIM: Mathematical Modelling and
  Numerical Analysis \textbf{56} (2022), no.~6, 2081--2103.

\bibitem{chen2020nonconforming}
Long Chen and Xuehai Huang, \emph{Nonconforming virtual element method for
  $2m$th order partial differential equations in $\mathbb{R}^n$}, Mathematics
  of Computation \textbf{89} (2020), no.~324, 1711--1744.

\bibitem{chen2021geometric}
\bysame, \emph{Geometric decompositions of the simplicial lattice and smooth
  finite elements in arbitrary dimension}, arXiv preprint arXiv:2111.10712
  (2021).

\bibitem{ciarlet1978finite}
P.~G. Ciarlet, \emph{The finite element method for elliptic problems},
  North-Holland, 1978.

\bibitem{da2020c1}
L~Beirao da~Veiga, Franco Dassi, and Alessandro Russo, \emph{{A $C^1$ virtual
  element method on polyhedral meshes}}, Computers \& Mathematics with
  Applications \textbf{79} (2020), no.~7, 1936--1955.

\bibitem{dai2013geometric}
Shibin Dai and Keith Promislow, \emph{{Geometric evolution of bilayers under
  the functionalized Cahn--Hilliard equation}}, Proceedings of the Royal
  Society A, vol. 469, The Royal Society, 2013, p.~20120505.

\bibitem{doelman2014meander}
A.~Doelman, G.~Hayrapetyan, K.~Promislow, and B.~Wetton, \emph{{Meander and
  pearling of single-curvature bilayer interfaces in the functionalized
  Cahn--Hilliard equation}}, SIAM Journal on Mathematical Analysis \textbf{46}
  (2014), no.~6, 3640--3677.

\bibitem{du2004phase}
Q.~Du, C.~Liu, and X.~Wang, \emph{A phase field approach in the numerical study
  of the elastic bending energy for vesicle membranes}, Journal of
  Computational Physics \textbf{198} (2004), no.~2, 450--468.

\bibitem{gallistl2017stable}
Dietmar Gallistl, \emph{Stable splitting of polyharmonic operators by
  generalized {S}tokes systems}, Mathematics of Computation \textbf{86} (2017),
  no.~308, 2555--2577.

\bibitem{gudi2011interior}
T.~Gudi and M.~Neilan, \emph{{An interior penalty method for a sixth-order
  elliptic equation}}, IMA Journal of Numerical Analysis \textbf{31} (2011),
  no.~4, 1734--1753.

\bibitem{hu2023construction}
Jun Hu, Ting Lin, and Qingyu Wu, \emph{{A construction of $C^r$ conforming
  finite element spaces in any dimension}}, Foundations of Computational
  Mathematics (2023), 1--37.

\bibitem{hu2024condition}
Jun Hu, Ting Lin, Qingyu Wu, and Beihui Yuan, \emph{The condition for
  constructing a finite element from a superspline}, arXiv preprint
  arXiv:2407.03680 (2024).

\bibitem{hu2014new}
Jun Hu, Rui Ma, and Zhong~Ci Shi, \emph{A new a priori error estimate of
  nonconforming finite element methods}, Science China Mathematics \textbf{57}
  (2014), no.~5, 887--902.

\bibitem{hu2020family}
Jun Hu, Shudan Tian, and Shangyou Zhang, \emph{{A family of 3D
  $H^2$-nonconforming tetrahedral finite elements for the biharmonic
  equation}}, Science China Mathematics \textbf{63} (2020), 1505--1522.

\bibitem{hu2015minimal}
Jun Hu and Shangyou Zhang, \emph{{The minimal conforming $H^k$ finite element
  spaces on $\mathcal{R}^n$ rectangular grids}}, Mathematics of Computation
  \textbf{84} (2015), no.~292, 563--579.

\bibitem{hu2016canonical}
\bysame, \emph{{A canonical construction of $H^m$-nonconforming triangular
  finite elements}}, Annals of Applied Mathematics \textbf{33} (2017), no.~3,
  266--288.

\bibitem{hu2009stable}
Zhengzheng Hu, Steven~M Wise, Cheng Wang, and John~S Lowengrub, \emph{Stable
  and efficient finite-difference nonlinear-multigrid schemes for the phase
  field crystal equation}, Journal of Computational Physics \textbf{228}
  (2009), no.~15, 5323--5339.

\bibitem{huang2020nonconforming}
Xuehai Huang, \emph{Nonconforming virtual element method for $2m$th order
  partial differential equations in $\mathbb{R}^n$ with $m>n$}, Calcolo
  \textbf{57} (2020), no.~4, 42.

\bibitem{li2024construction}
Jia Li and Shuonan Wu, \emph{A construction of canonical nonconforming finite
  element spaces for elliptic equations of any order in any dimension}, arXiv
  preprint arXiv:2409.06134 (2024).

\bibitem{li2014new}
Mingxia Li, Xiaofei Guan, and Shipeng Mao, \emph{{New error estimates of the
  Morley element for the plate bending problems}}, Journal of Computational and
  Applied Mathematics \textbf{263} (2014), 405--416.

\bibitem{mao2010error}
Shipeng Mao and Zhong~Ci Shi, \emph{On the error bounds of nonconforming finite
  elements}, Science China Mathematics \textbf{53} (2010), no.~11, 2917--2926.

\bibitem{schedensack2016new}
Mira Schedensack, \emph{A new discretization for {$m$th-Laplace} equations with
  arbitrary polynomial degrees}, SIAM Journal on Numerical Analysis \textbf{54}
  (2016), no.~4, 2138--2162.

\bibitem{scott1990finite}
L~Ridgway Scott and Shangyou Zhang, \emph{Finite element interpolation of
  nonsmooth functions satisfying boundary conditions}, Mathematics of
  Computation \textbf{54} (1990), no.~190, 483--493.

\bibitem{shi1990error}
Zhong~Ci Shi, \emph{{Error estimates of Morley element}}, Math. Numer. Sinica
  \textbf{12} (1990), no.~2, 113--118.

\bibitem{walkington2014c}
Noel~J Walkington, \emph{A {$C^1$} tetrahedral finite element without edge
  degrees of freedom}, SIAM Journal on Numerical Analysis \textbf{52} (2014),
  no.~1, 330--342.

\bibitem{wang2011energy}
Cheng. Wang and Steven~M Wise, \emph{An energy stable and convergent
  finite-difference scheme for the modified phase field crystal equation}, SIAM
  Journal on Numerical Analysis \textbf{49} (2011), no.~3, 945--969.

\bibitem{wang2001necessity}
Ming Wang, \emph{On the necessity and sufficiency of the patch test for
  convergence of nonconforming finite elements}, SIAM Journal on Numerical
  Analysis \textbf{39} (2001), no.~2, 363--384.

\bibitem{wang2013minimal}
Ming Wang and Jinchao Xu, \emph{{Minimal finite element spaces for
  $2m$-th-order partial differential equations in $\mathbb{R}^n$}}, Mathematics
  of Computation \textbf{82} (2013), no.~281, 25--43.

\bibitem{wise2009energy}
Steven~M Wise, Cheng Wang, and John~S Lowengrub, \emph{An energy-stable and
  convergent finite-difference scheme for the phase field crystal equation},
  SIAM Journal on Numerical Analysis \textbf{47} (2009), no.~3, 2269--2288.

\bibitem{wu2017nonconforming}
Shuonan Wu and Jinchao Xu, \emph{Nonconforming finite element spaces for
  $2m$-th order partial differential equations on $\mathbb{R}^n$ simplicial
  grids when $m= n+1$}, Mathematics of Computation \textbf{88} (2019), no.~316,
  531--551.

\bibitem{vzenivsek1970interpolation}
Alexander {\v{Z}}en{\'\i}{\v{s}}ek, \emph{Interpolation polynomials on the
  triangle}, Numerische Mathematik \textbf{15} (1970), no.~4, 283--296.

\bibitem{zhang2009family}
Shangyou Zhang, \emph{{A family of 3D continuously differentiable finite
  elements on tetrahedral grids}}, Applied Numerical Mathematics \textbf{59}
  (2009), no.~1, 219--233.

\end{thebibliography}

\end{document}